\documentclass[12pt,reqno]{amsart}
\usepackage{}
\usepackage{amsmath}
\usepackage{mathrsfs}
\usepackage{amssymb}
\usepackage{amsthm}
\usepackage{mathrsfs}
\usepackage[centertags]{amsmath}
\usepackage{amsfonts}
\usepackage[numbers,sort&compress]{natbib}
\usepackage{color}
\usepackage{extarrows}
\usepackage{fullpage}
\usepackage[colorlinks=true,  linkcolor=blue, citecolor=blue]{hyperref}
\input amssym.def
\input amssym.tex

\date{}

\newtheorem{Theorem}{Theorem}[section]

\newtheorem{Corollary}{Corollary}[section]
\newtheorem{Lemma}{Lemma}[section]

\newcommand\R{\mbox{\bf R}}

\newcommand\Z{\mbox{\bf Z}}

\newcommand\SR{\mbox{\scriptsize\bf R}}

\newcommand{\definition}{{\lower .5ex
  \hbox{$\>\>\stackrel{\triangle}{=}\>\>$} }}
\newcommand\supp{\mathop{\rm supp}}



\allowdisplaybreaks
\begin{document}

\baselineskip=22pt
\thispagestyle{empty}

\begin{center}
{\Large \bf  Almost sure spatial decay and  almost sure nonlinear
smoothing of  some stochastic dispersive equations}\\[1ex]

{Xiangqian Yan\footnote{Email: yanxiangqian213@126.com}$^{a}$,\,Yongsheng Li\footnote{Email: yshli@scut.edu.cn}$^{a}$,\,Jianhua Huang\footnote{Email:  jhhuang32@nudt.edu.cn}$^{b}$,\,Wei Yan\footnote{Email: 011133@htu.edu.cn}$^{c*}$}\\[1ex]

{$^a$School of Mathematics,
 South China University of Technology,}\\
 {Guangzhou, Guangdong 510640, P. R. China}\\[1ex]

{$^b$College of Sciences, National University of Defense Technology,}\\
{Changsha, Hunan 410073, P. R.  China}\\[1ex]

{$^{c*}$School of Mathematics and Statistics, Henan
Normal University,}\\
{Xinxiang, Henan 453007, P. R.  China}\\[1ex]

\end{center}
\noindent{\bf Abstract.}
In this paper, we consider the almost sure nonlinear smoothing, the almost sure spatial decay and the almost sure uniform convergence of  the stochastic mKdV equation and the stochastic cubic KdV-Benjamin-Ono equation.
Firstly, for initial data $g\in H^{s}(\R)(s\geq\frac{1}{4})$ and $\Phi_{2}\in L_{2}^{0,s}$, we prove the local well-posedness for the stochastic cubic KdV-Benjamin-Ono equation.
Secondly, we establish the almost sure nonlinear smoothing of the stochastic mKdV equation and the stochastic cubic KdV-Benjamin-Ono equation.
Finally, by using the almost sure nonlinear smoothing, we obtain the almost sure spatial decay and the almost sure uniform convergence of the integral term in the pathwise solutions to the stochastic mKdV equation and the stochastic cubic KdV-Benjamin-Ono equation. More precisely, we have the following results: for the stochastic mKdV equation, let $s>\frac{1}{3}$, $f\in H^{s}(\R)$ and $\Phi_{1}\in L_{2}^{0,s}$. Then, the local pathwise solution $u$  satisfies
\begin{eqnarray*}
&&\mathbb{P}\Big(\Big\{\omega: \lim_{t\rightarrow0}\Big\|u-U(t)f-\int_{0}^{t}U(t-s)\Phi_{1}dW(s)\Big\|_{L_{x}^{\infty}}=0\Big\}\Big)=1,\\
&&\mathbb{P}\Big(\Big\{\omega: \forall t\in[0,T_{\omega}], \lim_{|x|\rightarrow\infty}\Big(u-U(t)f-\int_{0}^{t}U(t-s)\Phi_{1}dW(s)\Big)=0\Big\}\Big)=1.
\end{eqnarray*}
For the stochastic cubic KdV-Benjamin-Ono equation, let $s>\frac{1}{3}$, $g\in H^{s}(\R)$ and $\Phi_{2}\in L_{2}^{0,s}$. Then, the local pathwise solution $v$ satisfies
\begin{eqnarray*}
&&\mathbb{P}\Big(\Big\{\omega: \lim_{t\rightarrow0}\Big\|v-V(t)g-\int_{0}^{t}V(t-s)\Phi_{2}dW(s)\Big\|_{L_{x}^{\infty}}=0\Big\}\Big)=1,\\
&&\mathbb{P}\Big(\Big\{\omega: \forall t\in[0,T_{\omega}],\lim_{|x|\rightarrow\infty}\Big(v-V(t)g-\int_{0}^{t}V(t-s)\Phi_{2}dW(s)\Big)=0\Big\}\Big)=1.
\end{eqnarray*}

 \bigskip

\medskip

\medskip
\noindent {\bf Corresponding Author:} Wei Yan

\medskip
\noindent {\bf Email Address:} 011133@htu.edu.cn

\medskip
\noindent {\bf MSC2020-Mathematics Subject Classification}: 60H15; 35Q53

\leftskip 0 true cm \rightskip 0 true cm

\newpage

\baselineskip=20pt

\bigskip
\bigskip
\tableofcontents

\section{Introduction and the main results}
\bigskip

\setcounter{Theorem}{0} \setcounter{Lemma}{0}\setcounter{Definition}{0}\setcounter{Proposition}{1}

\setcounter{section}{1}

\subsection{Research background}
In this paper, we consider the almost sure nonlinear smoothing, the almost sure spatial decay
and the almost sure uniform convergence for the stochastic mKdV equation
\begin{eqnarray}
&&\frac{du}{dt}+\partial_{x}^{3}u+\partial_{x}(u^{3})=\Phi_{1} \frac{d W}{dt},\label{1.01}\\
&&u(0,x)=f(x),\label{1.02}
\end{eqnarray}
and the stochastic cubic KdV-Benjamin-Ono(KdV-BO) equation
\begin{eqnarray}
&&\frac{dv}{dt}+\nu\mathscr{H}(\partial_{x}^{2}v)+\mu\partial_{x}^{3}v
+\partial_{x}(v^{3})=\Phi_{2}\frac{d W}{dt},\label{1.03}\\
&&v(0,x)=g(x),\label{1.04}
\end{eqnarray}
where
\begin{eqnarray*}
&&W(t)=\sum_{i=1}^{\infty}\beta_{i}e_{i}
\end{eqnarray*}
is a cylindrical Wiener process on $L^{2}$. $(e_{j}(x))_{j=1}^{\infty}$ is an orthonormal
basis of  $L^{2}(\R)$, and $(\beta_{j})_{j=1}^{\infty}$ is a sequence of mutually independent
 real Brownian motions on a fixed probability space. $\Phi_{1}$ and $\Phi_{2}$ are linear
 operators on $L^{2}$. The constants $\nu, \mu\in\R$ satisfy $\mu\neq0$; without loss of
 generality we take $\mu=1$ throughout the article. Finally, $\mathscr{H}$ denotes the Hilbert
 transform, which is defined as follows
\begin{eqnarray*}
&&\mathscr{H}f(x)=p.v.\frac{1}{\pi}\int\frac{f(x-y)}{y}dy.
\end{eqnarray*}

By choosing suitable function spaces and Banach fixed point argument,
Kenig et al. \cite{KPV1993} established the local well-posedness of  mKdV  in $H^{s}(\R)$
with $s\geq\frac{1}{4}$.
  Kenig et al. \cite{KPV2001} established the ill-posednes of mKdV
 in  $H^{s}(\R)$ with $s < \frac{1}{4}$
 in the sense   that the solution map
 fails to be uniformly continuous. By using the Fourier restriction norm method and Banach
 fixed point argument,
 Guo and Huo \cite{GH2004} established the local well-posedness   of the Cauchy problem for
 the cubic KdV-BO equation for initial data in $H^{s}(\R)$ with $s\geq\frac{1}{4}$.

 de Bouard and  Debussche \cite{BD1998} employed the martingale solution and compactness method
  to establish the well-posedness framework for the stochastic KdV equation with multiplicative
  noise. By employing the Fourier restriction norm method introduced by Bourgain \cite{B1993},
 stochastic convolution estimates, the fixed-point theorem, and stopping time techniques,
 de Bouard    et al. \cite{BDT1999}
  established the well-posedness of the Cauchy problem for the  stochastic KdV equation with
  additive noise in the  $L^{2}(\R)$.

Recently,
for $f\in H^{\frac{1}{4}}(\R)$ and $\Phi_{1}\in L_{2}^{0,1+\epsilon}$, by using the
suitable function spaces,  Millet and Roudenko
\cite{MR2018} established the local well-posedness of the
Cauchy problem for the white noise-driven stochastic mKdV equation, and they also proved the
 global well-posedness of the Cauchy problem for the stochastic mKdV and the stochastic quartic
  KdV  with the additive noise in the $H^{1}(\R)$. Very recently, for $f\in H^{s}(\R)(s\geq\frac{1}{4})$
  and $\Phi_{1}\in L_{2}^{0,s}$, by using the
  method developed by de Bouard et al. \cite{BDT1999}, Yan et al. \cite{YHG2021} established the local well-posedness of the
Cauchy problem for the white noise-driven stochastic mKdV equation. Moreover, when $f\in H^{1}(\R)$ and $\Phi_1 \in L_2^{0,1}$,
they also proved the
 global well-posedness of the Cauchy problem for the stochastic mKdV.
By applying the method developed by de Bouard et al. \cite{BDT1999}, Wang and Guo \cite{WG2010}
 established the
well-posedness theory of the Cauchy problem for the stochastic KdV-BO equation driven by white
noise. However, the Cauchy  problem for the
stochastic cubic KdV-BO equation has not been systematically studied.

\subsection{Motivations}
Our motivation stems primarily from the following two considerations. On the one hand,
Correia and Silva \cite{CS2020} established the nonlinear smoothing effect for the
deterministic mKdV equation. This prompts a natural question: can we obtain the almost
 sure nonlinear smoothing for the stochastic mKdV equation and the stochastic cubic KdV-BO equation?
  On the other hand, Yan et al. \cite{YYY2026,YLHHY2026} established the spatial decay property and
  the uniform convergence of the Cauchy problem for the generalized KdV equation and the generalized Ostrovsky equation.
  This leads to a parallel question: can we
  obtain the almost sure spatial decay and the  almost sure uniform convergence of the Cauchy problem for the stochastic
  mKdV equation and the stochastic cubic KdV-BO equation?
\subsection{Contributions}
In this paper, we provide affirmative answers to the above  questions.
Firstly, by using the method developed by de Bouard et al. \cite{BDT1999}  and assuming that the
initial data satisfy $g\in H^{s}(\R)(s\geq\frac{1}{4})$ and $\Phi_{2}\in L_{2}^{0,s}$,
 we establish the local well-posedness of the Cauchy problem for  the stochastic cubic
 KdV-BO equation.
Secondly, by using the Fourier restriction norm method introduced in \cite{B1983,B1993,KM1993,RR1982} and the high-low frequency
decomposition method, we establish the almost sure nonlinear smoothing  for the stochastic mKdV equation
 and the stochastic cubic KdV-BO equation. More precisely, we have the following results.
  Let $s\geq\frac{1}{4}$,\,$0<\epsilon\leq\frac{1}{2026}$. For the stochastic mKdV equation, with  $b=\frac{1}{2}+\epsilon$, $0\leq a\leq
\min\left\{2s-\frac{1}{2},1-64\epsilon\right\}$, $\Phi_{1}\in L_{2}^{0,s}$
and $f\in H^{s}(\R)$, we have
\begin{eqnarray*}
\mathbb{P}\left(\left\{\omega:  \|u_{1}\|_{X_{s+a,b,\varphi}^{T_{\omega}}}<\infty\right\}\right)=1,
\end{eqnarray*}
where the definition of $u_{1}$ can be found in \eqref{1.06}.
For the stochastic cubic KdV-BO equation, with  $b=\frac{1}{2}+\epsilon$,
$0\leq a\leq\min\left\{2s-\frac{1}{2},1-48\epsilon\right\}$, $\Phi_{2}\in L_{2}^{0,s}$
and $g\in H^{s}(\R)$, we have
\begin{eqnarray*}
\mathbb{P}\left(\left\{\omega: \|v_{1}\|_{X_{s+a,b,\phi}^{T_{\omega}}}<\infty\right\}\right)=1,
\end{eqnarray*}
where the definition of $v_{1}$ can be found in \eqref{1.011}.
Finally, by using the almost sure nonlinear smoothing, we obtain the almost sure spatial decay and the almost
sure uniform convergence of the stochastic mKdV equation and the stochastic cubic KdV-BO equation. More precisely,
we have the following results. Let $s>\frac{1}{3}$. For the stochastic mKdV equation, with  $\Phi_{1}\in L_{2}^{0,s}$
and $f\in H^{s}(\R)$, we have
\begin{eqnarray*}
&&\mathbb{P}\left(\left\{\omega: \lim_{t\rightarrow0}\left\|u_{1}\right\|_{L_{x}^{\infty}}=0\right\}\right)=1,\quad
\mathbb{P}\left(\left\{\omega: \forall t\in[0,T_{\omega}], \lim_{|x|\rightarrow\infty}u_{1}=0\right\}\right)=1,
\end{eqnarray*}
where the definition of $u_{1}$ can be found in \eqref{1.06}.
For the stochastic cubic KdV-BO equation, with $\Phi_{2}\in L_{2}^{0,s}$ and $g\in H^{s}(\R)$, we have
\begin{eqnarray*}
&&\mathbb{P}\left(\left\{\omega: \lim_{t\rightarrow0}\left\|v_{1}\right\|_{L_{x}^{\infty}}=0\right\}\right)=1,\quad
\mathbb{P}\left(\left\{\omega: \forall t\in[0,T_{\omega}],\lim_{|x|\rightarrow\infty}v_{1}=0\right\}\right)=1,
\end{eqnarray*}
where the definition of $v_{1}$ can be found in \eqref{1.011}.

\subsection{Introduction to notations and spaces}
Before stating the main results, we introduce the notation that will be used throughout the proofs.
$a\sim b$ means that there exists $C_{1}, C_{2}>0$ satisfying $C_{1}|a|\leq|b|\leq C_{2}|a|$. Let $\psi(t)$ be
a function in $C_{c}^{\infty}(\R)$ with $\psi(t)=1, t\in[0,1]$, $\psi(t)=0$, $t\geq2$.
We define $A:=\max\left\{1,\frac{2|\nu|}{3}\right\}$, $ \langle x\rangle:=(1+|x|^{2})^{\frac{1}{2}}$, $\phi(\xi):=-\nu\xi|\xi|+\xi^{3}$ and $\varphi(\xi):=\xi^{3}$,
\begin{eqnarray*}
&&\mathscr{F}_{x}f(\xi)=\frac{1}{\sqrt{2\pi}}\int_{\SR}e^{-ix\xi}f(x)dx,\, \mathscr{F}_{\xi}^{-1}f(x)=\frac{1}{\sqrt{2\pi}}\int_{\SR}e^{ix\xi}\mathscr{F}_{x}f(\xi)d\xi,\\
&&\mathscr{F}_{xt}f(\xi,\tau)=\frac{1}{2\pi}\int_{\SR^{2}}e^{-it\tau}e^{-ix\xi}f(x,t)dxdt,\, \mathscr{F}_{\xi\tau}^{-1}f(x,t)=\frac{1}{2\pi}\int_{\SR^{2}}e^{it\tau}e^{ix\xi}\mathscr{F}_{xt}f(\xi,\tau)d\xi d\tau,\\
&&P_{\leq 2A}f=\frac{1}{\sqrt{2\pi}}\int_{|\xi|\leq 2A}e^{ix\xi}\mathscr{F}_{x}f(\xi)d\xi,\,
 P_{\geq 2A}f=\frac{1}{\sqrt{2\pi}}\int_{|\xi|\geq 2A}e^{ix\xi}\mathscr{F}_{x}f(\xi)d\xi,\\
&&U(t)f=\frac{1}{\sqrt{2\pi}}\int_{\SR}e^{ix\xi}e^{it\varphi(\xi)}\mathscr{F}_{x}f(\xi)d\xi,\, V(t)g=\frac{1}{\sqrt{2\pi}}\int_{\SR}e^{ix\xi}e^{it\phi(\xi)}\mathscr{F}_{x}g(\xi)d\xi.
\end{eqnarray*}
The spaces $X_{s,b,\varphi}(\R^{2})$ and $X_{s,b,\phi}(\R^{2})$ are defined as follows
\begin{eqnarray*}
&&X_{s,b,\varphi}(\R^{2})=\{u\in\mathcal{S}^{\prime}(\R^{2}):\|u\|_{X_{s,b}}=
\left\|\langle\tau-\varphi(\xi)\rangle^{b}\langle\xi\rangle^{s}\mathscr{F}_{xt}
u(\tau,\xi)\right\|_{L_{\tau\xi}^{2}}<\infty\},\\
&&X_{s,b,\phi}(\R^{2})=\{v\in\mathcal{S}^{\prime}(\R^{2}):\|v\|_{X_{s,b}}=
\left\|\langle\tau-\phi(\xi)\rangle^{b}\langle\xi\rangle^{s}\mathscr{F}_{xt}
v(\tau,\xi)\right\|_{L_{\tau\xi}^{2}}<\infty\}.
\end{eqnarray*}
For $T>0$, we denote by $X_{s,b,\varphi}^{T}$ and $X_{s,b,\phi}^{T}$ the restrictions
of $X_{s,b,\varphi}$ and $X_{s,b,\phi}$ to $\R\times[0,T]$, respectively, equipped with the following norms
\begin{eqnarray*}
&&\|u\|_{X_{s,b,\varphi}^{T}}=\inf\{\|\bar{u}\|_{X_{s,b,\varphi}}, \bar{u}=u, t\in[0,T]\},\\
&&\|u\|_{X_{s,b,\phi}^{T}}=\inf\{\|\bar{v}\|_{X_{s,b,\phi}}, \bar{v}=v, t\in[0,T]\}.
\end{eqnarray*}
Let $(e_{i})_{i\in\mathbf{N}}$ be an orthonormal basis of $L^{2}(\R)$. We denote by $L_{2}^{0,s}$
the space of all Hilbert-Schmidt operators from $L^2(\R)$ to $H^{s}(\R)$, endowed with the norm
\begin{eqnarray*}
&&\|\Phi\|_{L_{2}^{0,s}}=\left(\sum_{i\in\mathbf{N}}\|\Phi e_{i}\|_{H^{s}}^{2}\right)^{\frac{1}{2}}.
\end{eqnarray*}

\subsection{The main results }

\begin{Theorem}\label{Theorem1}(Almost sure nonlinear smoothing for the stochastic mKdV equation)
Let $s\geq\frac{1}{4}$, $f\in H^{s}(\R)$,\,$0<\epsilon\leq\frac{1}{2026}$,\, $b=\frac{1}{2}+\epsilon$ and $0\leq a\leq
\min\left\{2s-\frac{1}{2},1-64\epsilon\right\}$, $\Phi_{1}\in L_{2}^{0,s}$,
and let $u$ be the pathwise solution to the stochastic mKdV equation.
Then,  we have
\begin{eqnarray}
\mathbb{P}\left(\left\{\omega:  \|u_{1}\|_{X_{s+a,b,\varphi}^{T_{\omega}}}<\infty\right\}\right)=1,\label{1.05}
\end{eqnarray}
where
\begin{eqnarray}
&&u_{1}=u-U(t)f-w_{1}=-\int_{0}^{t}U(t-s)\partial_{x}\left[(u_{1}+U(t)f+w_{1})^{3}\right]ds,\label{1.06}\\
&&w_{1}=\int_{0}^{t}U(t-s)\Phi_{1}dW(s).\label{1.07}
\end{eqnarray}
\end{Theorem}

\begin{Theorem}\label{Theorem2}(Almost sure uniform convergence for the stochastic mKdV equation)
Assume that $s>\frac{1}{3}$, $f\in H^{s}(\R)$, $\Phi_{1}\in L_{2}^{0,s}$ and let $u_{1}$
be defined as in \eqref{1.06}. Then, we have
\begin{eqnarray}
\mathbb{P}\left(\left\{\omega: \lim_{t\rightarrow0}\|u_{1}\|_{L_{x}^{\infty}}=0\right\}\right)=1.\label{1.08}
\end{eqnarray}

\end{Theorem}

\begin{Theorem}\label{Theorem3}(Almost sure spatial decay for the stochastic mKdV equation)
Assume that $s>\frac{1}{3}$, $f\in H^{s}(\R)$, $\Phi_{1}\in L_{2}^{0,s}$ and let $u_{1}$
be defined as in \eqref{1.06}. Then, we have
\begin{eqnarray}
\mathbb{P}\left(\left\{\omega: \forall t\in[0,T_{\omega}], \lim_{|x|\rightarrow\infty}u_{1}=0\right\}\right)=1.\label{1.09}
\end{eqnarray}

\end{Theorem}

\begin{Theorem}\label{Theorem4}(Almost sure local well-posedness for the stochastic cubic KdV-BO equation)
Assume that $s\geq\frac{1}{4}$, $g\in H^{s}(\R)$, $\Phi_{2}\in L_{2}^{0,s}$ and $b_{1}=\frac{1}{2}-3\epsilon$
with $0<\epsilon\leq\frac{1}{2026}$. Then,
there exists there exists a stopping time $T_{\omega}$, such that
\begin{eqnarray*}
&&\mathbb{P}\left(\{\omega: 0<T_{\omega}<1\}\right)=1,
\end{eqnarray*}
and the equation \eqref{1.03}-\eqref{1.04} admits a unique solution $v$ satisfying
\begin{eqnarray*}
&&\mathbb{P}\left(\left\{\omega: v\in C([0,T_{\omega}], H^{s}(\R))\cap X_{s,b_{1},\phi}^{T_{\omega}}\right\}\right)=1.
\end{eqnarray*}

\end{Theorem}

\begin{Theorem}\label{Theorem5}(Almost sure nonlinear smoothing for the stochastic cubic KdV-BO equation)
Let $s\geq\frac{1}{4}$,\,$g\in H^{s}(\R)$, $0<\epsilon\leq\frac{1}{2026}$,\, $b=\frac{1}{2}+\epsilon$ and $0\leq a\leq\min\left\{2s-\frac{1}{2},1-48\epsilon\right\}$, $\Phi_{2}\in L_{2}^{0,s}$, and let $v$ be the pathwise solution to the stochastic cubic KdV-BO equation.
Then,  we have
\begin{eqnarray}
\mathbb{P}\left(\left\{\omega: \|v_{1}\|_{X_{s+a,b,\phi}^{T_{\omega}}}<\infty\right\}\right)=1,\label{1.010}
\end{eqnarray}
where
\begin{eqnarray}
&&v_{1}=v-V(t)g-w_{2}=-\int_{0}^{t}V(t-s)\partial_{x}\left[(v_{1}+V(t)g+w_{2})^{3}\right]ds,\label{1.011}\\
&&w_{2}=\int_{0}^{t}V(t-s)\Phi_{2}dW(s).\label{1.012}
\end{eqnarray}
\end{Theorem}

\begin{Theorem}\label{Theorem6}(Almost sure uniform convergence for the stochastic cubic KdV-BO equation)
Let $s>\frac{1}{3}$, $g\in H^{s}(\R)$, $\Phi_{2}\in L_{2}^{0,s}$ and let $v_{1}$
be defined as in \eqref{1.011}. Then, we have
\begin{eqnarray}
\mathbb{P}\left(\left\{\omega: \lim_{t\rightarrow0}\|v_{1}\|_{L_{x}^{\infty}}=0\right\}\right)=1.\label{1.013}
\end{eqnarray}

\end{Theorem}

\begin{Theorem}\label{Theorem7}(Almost sure spatial decay for the stochastic cubic KdV-BO equation)
Let $s>\frac{1}{3}$, $g\in H^{s}(\R)$, $\Phi_{2}\in L_{2}^{0,s}$ and let $v_{1}$
be defined as in \eqref{1.011}. Then, we have
\begin{eqnarray}
\mathbb{P}\left(\left\{\omega: \forall t\in[0,T_{\omega}],\lim_{|x|\rightarrow\infty} v_{1}=0\right\}\right)=1.\label{1.014}
\end{eqnarray}

\end{Theorem}

The remainder of this paper is organized as follows. In Section 2,  we give
some lemmas related to $X_{s,b}$. In Section 3,
 we give some lemmas related to the stochastic convolution, which play a key role in proving that $T_{\omega}$ is a stopping time.
 In Section 4,
 we prove the trilinear estimates, which are key to proving Theorems 1.1-1.7. In Section 5, we prove
 Theorem 1.1. In Section 6, we prove
 Theorem 1.2. In Section 7, we prove
 Theorem 1.3.  In Section 8, we prove
 Theorem 1.4. In Section 9, we prove
 Theorem 1.5.  In Section 10, we prove
 Theorem 1.6. In Section 11, we prove
 Theorem 1.7.

\bigskip
\bigskip

\section{Estimates Related to $X_{s,b}$}

\setcounter{equation}{0}

\setcounter{Theorem}{0}

\setcounter{Lemma}{0}

\setcounter{Proposition}{0}

\setcounter{Corollary}{0}

\setcounter{section}{2}
In this section, we present some lemmas related to $X_{s,b}$, which will be used in the subsequent proofs.

\begin{Lemma}\label{lem2.1}
Let $b=\frac{1}{2}+\epsilon$, $0<\epsilon\leq\frac{1}{2026}$, we have
\begin{eqnarray}
&&\left\|D_{x}^{\frac{1}{6}}u\right\|_{L_{xt}^{6}}\leq C\|u\|_{X_{0,b,\varphi}},\label{2.01}\\
&&\left\|D_{x}^{\frac{1}{8}}u\right\|_{L_{xt}^{4}}\leq C\|u\|_{X_{0,\frac{3}{4}b,\varphi}},\label{2.02}\\
&&\left\|u\right\|_{L_{xt}^{8}}\leq C\|u\|_{X_{0,b,\varphi}},\label{2.03}\\
&&\left\|u\right\|_{L_{xt}^{4}}\leq C\|u\|_{X_{0,\frac{2}{3}b,\varphi}}.\label{2.04}
\end{eqnarray}

\end{Lemma}
\noindent{\bf Proof.}
For \eqref{2.01} and \eqref{2.03}, we refer to \cite[Theorem 2.1]{KPV1993Duke}, respectively.
By applying the interpolation theorem in mixed Lebesgue spaces \cite[Theorem 1]{BP1961} to \eqref{2.01} and
\begin{eqnarray}
&&\|u\|_{L_{xt}^{2}}\leq C\|u\|_{X_{0,0,\varphi}},\label{2.05}
\end{eqnarray}
we obtain \eqref{2.02}. By applying the interpolation theorem in mixed Lebesgue spaces
\cite[Theorem 1]{BP1961} to \eqref{2.03} and \eqref{2.05}, we obtain \eqref{2.04}.

The proof of Lemma 2.1 is completed.

\begin{Lemma}\label{lem2.2}
Let $b=\frac{1}{2}+\epsilon$, $0<\epsilon\leq\frac{1}{2026}$, $0\leq s\leq\frac{1}{2}$, we have
\begin{eqnarray}
&&\left\|I^{s}(f,g)\right\|_{L_{xt}^{2}}\leq C\|f\|_{X_{0,\frac{3+2s}{4}b,\varphi}}\|g\|_{X_{0,\frac{3+2s}{4}b,\varphi}},\label{2.06}
\end{eqnarray}
where
\begin{eqnarray*}
&&\mathscr{F}_{xt}I^{s}(f,g)=\int_{\substack{
\xi=\xi_{1}+\xi_{2}\\
\tau=\tau_{1}+\tau_{2}}}\left|\xi_{1}^{2}-\xi_{2}^{2}\right|^{s}
\mathscr{F}_{xt}f(\xi_{1},\tau_{1})\mathscr{F}_{xt}g(\xi_{2},\tau_{2})d\xi_{1}d\tau_{1}.
\end{eqnarray*}
\end{Lemma}

For Lemma 2.2, we refer to \cite{G2002}.

By using Lemma 2.2, we have the following corollary.

\begin{Corollary}\label{cor2.1}
Let $b=\frac{1}{2}+\epsilon$, $0<\epsilon\leq\frac{1}{2026}$, $0\leq s\leq\frac{1}{2}$ and $N_{1}\gg N_{2}$.
Suppose that $\supp\mathscr{F}_{x}f\subset\{\xi:|\xi|\sim N_{1}\}$ and $\supp\mathscr{F}_{x}g\subset\{\xi:|\xi|\sim N_{2}\}$.
Then,  we have
\begin{eqnarray*}
&&\left\|fg\right\|_{L_{xt}^{2}}\leq CN_{1}^{-2s}\|f\|_{X_{0,\frac{3+2s}{4}b,\varphi}}\|g\|_{X_{0,\frac{3+2s}{4}b,\varphi}}.
\end{eqnarray*}

\end{Corollary}

\begin{Lemma}\label{lem2.3}
Let $b=\frac{1}{2}+\epsilon$, $0<\epsilon\leq\frac{1}{2026}$, we have
\begin{eqnarray}
&&\left\|v\right\|_{L_{xt}^{8}}\leq C\|v\|_{X_{0,b,\phi}},\label{2.07}\\
&&\left\|v\right\|_{L_{xt}^{4}}\leq C\|v\|_{X_{0,\frac{2}{3}b,\phi}},\label{2.08}\\
&&\left\|D_{x}^{\frac{1}{6}}v\right\|_{L_{xt}^{6}}\leq C\|v\|_{X_{0,b,\phi}},\label{2.09}\\
&&\left\|D_{x}^{\frac{1}{8}}v\right\|_{L_{xt}^{4}}\leq C\|v\|_{X_{0,\frac{3}{4}b,\phi}}.\label{2.010}
\end{eqnarray}

\end{Lemma}
\noindent{\bf Proof.}
For \eqref{2.07}, we refer to \cite[Lemma 2.1]{GH2004}. By applying the interpolation theorem in mixed Lebesgue spaces
\cite[Theorem 1]{BP1961} to \eqref{2.07} and \begin{eqnarray}
&&\|v\|_{L_{xt}^{2}}\leq C\|v\|_{X_{0,0,\phi}},\label{2.011}
\end{eqnarray}
we obtain \eqref{2.08}. For \eqref{2.09}, on the one hand, from \cite[(2.4)]{GH2004}, we have
\begin{eqnarray}
&&\left\|D_{x}^{\frac{1}{6}}P_{\geq2A}v\right\|_{L_{xt}^{6}}\leq C\|v\|_{X_{0,b,\phi}}.\label{2.012}
\end{eqnarray}
On the other hand, by using \eqref{2.07}, \eqref{2.011} and interpolation theorem, we have
\begin{eqnarray}
&&\left\|D_{x}^{\frac{1}{6}}P_{\leq2A}v\right\|_{L_{xt}^{6}}\leq C\left\|D_{x}^{\frac{1}{6}}P_{\leq2A}v\right\|_{L_{xt}^{2}}^{\frac{1}{9}}
\left\|D_{x}^{\frac{1}{6}}P_{\leq2A}v\right\|_{L_{xt}^{8}}^{\frac{8}{9}}\nonumber\\
&&\leq C\left\|D_{x}^{\frac{1}{6}}P_{\leq2A}v\right\|_{X_{0,b,\phi}}\leq C\|v\|_{X_{0,b,\phi}}.\label{2.013}
\end{eqnarray}
From \eqref{2.012} and \eqref{2.013}, we have
\begin{eqnarray}
&&\left\|D_{x}^{\frac{1}{6}}v\right\|_{L_{xt}^{6}}\leq C\left(\left\|D_{x}^{\frac{1}{6}}P_{\geq2A}v\right\|_{L_{xt}^{6}}+\left\|D_{x}^{\frac{1}{6}}P_{\leq2A}v\right\|_{L_{xt}^{6}}\right)
\leq C\|v\|_{X_{0,b,\phi}}.\label{2.014}
\end{eqnarray}
For \eqref{2.010}, by applying the interpolation theorem in mixed Lebesgue spaces
\cite[Theorem 1]{BP1961} to \eqref{2.09} and \eqref{2.011}, we obtain \eqref{2.010}.

The proof of Lemma 2.3 is completed.

\begin{Lemma}\label{lem2.4}
Let $b=\frac{1}{2}+\epsilon$, $0<\epsilon\leq\frac{1}{2026}$, $0\leq s\leq\frac{1}{2}$, $\phi(\xi)=-\nu\xi|\xi|+\xi^{3}$, we have
\begin{eqnarray}
&&\left\|K^{s}(f,g)\right\|_{L_{xt}^{2}}\leq C\|f\|_{X_{0,\frac{2(1+s)}{3}b,\phi}}\|g\|_{X_{0,\frac{2(1+s)}{3}b,\phi}},\label{2.015}
\end{eqnarray}
where
\begin{eqnarray*}
&&\mathscr{F}_{xt}K^{s}(f,g)=\int_{\substack{
\xi=\xi_{1}+\xi_{2}\\
\tau=\tau_{1}+\tau_{2}}}\left|\phi^{\prime}(\xi_{1})-\phi^{\prime}(\xi_{2})\right|^{s}
\mathscr{F}_{xt}f(\xi_{1},\tau_{1})\mathscr{F}_{xt}g(\xi_{2},\tau_{2})d\xi_{1}d\tau_{1}.
\end{eqnarray*}
\end{Lemma}

\noindent{\bf Proof.}
On the one hand, from \cite[Lemma 2.5]{LW2010}, we have
\begin{eqnarray}
&&\left\|K^{\frac{1}{2}}(f,g)\right\|_{L_{xt}^{2}}\leq C\|f\|_{X_{0,b,\phi}}\|g\|_{X_{0,b,\phi}},\label{2.016}
\end{eqnarray}
On the other hand, by using H\"{o}lder  inequality and \eqref{2.08}, we have
\begin{eqnarray}
&&\left\|fg\right\|_{L_{xt}^{2}}\leq C\|f\|_{L_{xt}^{4}}\|g\|_{L_{xt}^{4}}
\leq C\|f\|_{X_{0,\frac{2}{3}b,\phi}}\|g\|_{X_{0,\frac{2}{3}b,\phi}}.\label{2.017}
\end{eqnarray}
By applying the interpolation theorem in mixed Lebesgue spaces
\cite[Theorem 1]{BP1961} to \eqref{2.016} and \eqref{2.017}, we obtain \eqref{2.015}.

The proof of Lemma 2.4 is completed.

By using Lemma 2.4, we have the following corollary.

\begin{Corollary}\label{cor2.2}
Let $b=\frac{1}{2}+\epsilon$, $0<\epsilon\leq\frac{1}{2026}$, $0\leq s\leq\frac{1}{2}$ and $N_{1}\geq 2A$, $N_{1}\gg N_{2}$. Suppose that $\supp\mathscr{F}_{x}f\subset\{\xi:|\xi|\sim N_{1}\}$ and $\supp\mathscr{F}_{x}g\subset\{\xi:|\xi|\sim N_{2}\}$. Then,  we have
\begin{eqnarray*}
&&\left\|fg\right\|_{L_{xt}^{2}}\leq CN_{1}^{-2s}\|f\|_{X_{0,\frac{2(1+s)}{3}b,\phi}}\|g\|_{X_{0,\frac{2(1+s)}{3}b,\phi}}.
\end{eqnarray*}

\end{Corollary}

\begin{Lemma}\label{lem2.5}
Let $s\in\R$, $0<\epsilon\leq\frac{1}{2026}$, $0<T<1$, $0<b, b_{1}<\frac{1}{2}$, we have
\begin{eqnarray}
&&\left\|U(t)f\right\|_{X^{T}_{s,\frac{1}{2}+\epsilon,\varphi}}\leq C\|f\|_{H^{s}},\label{2.018}\\
&&\left\|\int_{0}^{t}U(t-t^{\prime})f(t^{\prime})dt^{\prime}\right\|_{X_{s,b,\varphi}^{T}}
\leq CT^{1-b_{1}-b}\|f\|_{X_{0,-b_{1},\varphi}^{T}},\label{2.019}\\
&&\left\|V(t)f\right\|_{X^{T}_{s,\frac{1}{2}+\epsilon,\phi}}\leq C\|f\|_{H^{s}},\label{2.020}\\
&&\left\|\int_{0}^{t}V(t-t^{\prime})f(t^{\prime})dt^{\prime}\right\|_{X_{s,b,\phi}^{T}}
\leq CT^{1-b_{1}-b}\|f\|_{X_{0,-b_{1},\phi}^{T}}.\label{2.021}
\end{eqnarray}
\end{Lemma}

For Lemma 2.5, we refer to \cite{KPV1993Duke, BDT1999}.

\begin{Lemma}\label{lem2.6}
Let $s\geq0$ and $0\leq b<\frac{1}{2}$. Then there exist constants $C_{1}, C_{2}$ such that
\begin{eqnarray*}
&&C_{1}\|u\|_{X_{s,b}^{T}}\leq \|\chi_{[0,T]}(\cdot)u\|_{X_{s,b}}\leq C_{2}\|u\|_{X_{s,b}^{T}},
\end{eqnarray*}
where $C_1$ and $C_2$ depend only on $b$ and are independent of $T$.

\end{Lemma}

For Lemma 2.6, we refer to \cite[Lemma 2.1]{BD2007}.

\begin{Lemma}\label{lem2.7}
Let $0\leq b<\frac{1}{2}$, and suppose that $\left|\mathscr{F}_{x}f(\xi,t)\right|\leq C$.
Then for any $T_{1}, T_{2}\in\R$ with $T_{1}<T_{2}$, we have
\begin{eqnarray}
\left\|\chi_{[T_{1},T_{2}]}(\cdot)\mathscr{F}_{x}f(\xi,t)\right\|_{H^{b}}\leq C_{1}\left((T_{2}-T_{1})^{\frac{1}{2}}
+(T_{2}-T_{1})^{\frac{1}{2}-b}+(T_{2}-T_{1})^{\frac{3}{2}-b}\right),\label{2.022}
\end{eqnarray}
where $C_{1}$ is a constant depending only on $b$.

\end{Lemma}
\noindent{\bf Proof.}
For $b=0$, we have
\begin{eqnarray}
&&\left\|\chi_{[T_{1},T_{2}]}(\cdot)\mathscr{F}_{x}f(\xi,t)\right\|_{L^{2}}
=\left(\int_{\SR}\left|\chi_{[T_{1},T_{2}]}(\cdot)\mathscr{F}_{x}f(\xi,t)\right|
^{2}dt\right)^{\frac{1}{2}}\nonumber\\
&&\leq C(T_{2}-T_{1})^{\frac{1}{2}}.\label{2.023}
\end{eqnarray}
Note that $H^{b}$ has the following equivalent norm
\begin{eqnarray}
&&\left\|\chi_{[T_{1},T_{2}]}(\cdot)\mathscr{F}_{x}f(\xi,t)\right\|_{H^{b}}^{2}
=\int\int_{\SR^{2}}\frac{|\chi_{[T_{1},T_{2}]}(\cdot)\mathscr{F}_{x}f(\xi,t)
-\chi_{[T_{1},T_{2}]}(\cdot)\mathscr{F}_{x}f(\xi,r)|^{2}}
{|t-r|^{1+2b}}dtdr\nonumber\\
&&\quad+\left\|\chi_{[T_{1},T_{2}]}(\cdot)\mathscr{F}_{x}f(\xi,t)\right\|_{L^{2}}^{2}=L_{1}+L_{2}.\label{2.024}
\end{eqnarray}
Here
\begin{eqnarray*}
&&L_{1}=\int\int_{\SR^{2}}\frac{|\chi_{[T_{1},T_{2}]}(\cdot)\mathscr{F}_{x}f(\xi,t)
-\chi_{[T_{1},T_{2}]}(\cdot)\mathscr{F}_{x}f(\xi,r)|^{2}}{|t-r|^{1+2b}}dtdr,\,\,
L_{2}=\left\|\chi_{[T_{1},T_{2}]}(\cdot)\mathscr{F}_{x}f(\xi,t)\right\|_{L^{2}}^{2}.
\end{eqnarray*}
For $L_{1}$, we have
\begin{eqnarray}
&&L_{1}=2\int\int_{t<r}\frac{|\chi_{[T_{1},T_{2}]}\mathscr{F}_{x}f(\xi,t)
-\chi_{[T_{1},T_{2}]}\mathscr{F}_{x}f(\xi,r)|^{2}}{|t-r|^{1+2b}}dtdr\nonumber\\
&&=2\int_{T_{1}}^{T_{2}}\int_{-\infty}^{T_{1}}
\frac{|\mathscr{F}_{x}f(\xi,r)|^{2}}{|t-r|^{1+2b}}dtdr+2\int_{T_{1}}^{T_{2}}\int_{T_{1}}^{r}
\frac{|\mathscr{F}_{x}f(\xi,t)-\mathscr{F}_{x}f(\xi,r)|^{2}}{|t-r|^{1+2b}}dtdr\nonumber\\
&&+2\int_{T_{2}}^{\infty}\int_{T_{1}}^{T_{2}}\frac{|\mathscr{F}_{x}f(\xi,t)|^{2}}{|t-r|^{1+2b}}dtdr\nonumber\\
&&=\sum_{i=1}^{3}I_{i}.\label{2.025}
\end{eqnarray}
Here
\begin{eqnarray*}
&&I_{1}=2\int_{T_{1}}^{T_{2}}\int_{-\infty}^{T_{1}}\frac{|\mathscr{F}_{x}f(\xi,r)|^{2}}{|t-r|^{1+2b}}dtdr,\\
&&I_{2}=2\int_{T_{1}}^{T_{2}}\int_{T_{1}}^{r}\frac{|\mathscr{F}_{x}f(\xi,t)-\mathscr{F}_{x}f(\xi,r)|^{2}}{|t-r|^{1+2b}}dtdr,\\
&&I_{3}=2\int_{T_{2}}^{\infty}\int_{T_{1}}^{T_{2}}\frac{|\mathscr{F}_{x}f(\xi,t)|^{2}}{|t-r|^{1+2b}}dtdr.
\end{eqnarray*}
For $I_{1}$, since $\left|\mathscr{F}_{x}f(\xi,t)\right|\leq C$, we have
\begin{eqnarray}
&&I_{1}=2\int_{T_{1}}^{T_{2}}\int_{-\infty}^{T_{1}}\frac{|\mathscr{F}_{x}f(\xi,r)|^{2}}{|t-r|^{1+2b}}dtdr\nonumber\\
&&\leq C\int_{T_{1}}^{T_{2}}\int_{-\infty}^{T_{1}}\frac{1}{|t-r|^{1+2b}}dtdr\nonumber\\
&&\leq C\int_{T_{1}}^{T_{2}}(r-T_{1})^{-2b}dr=C(T_{2}-T_{1})^{1-2b}.\label{2.026}
\end{eqnarray}
For $I_{2}$, arguing similarly to the estimate for $I_{1}$
in \eqref{2.030}, we have
\begin{eqnarray}
&&I_{2}=2\int_{T_{1}}^{T_{2}}\int_{T_{1}}^{r}\frac{|\mathscr{F}_{x}f(\xi,t)-\mathscr{F}_{x}f(\xi,r)|^{2}}{|t-r|^{1+2b}}dtdr\nonumber\\
&&\leq C\int_{T_{1}}^{T_{2}}\int_{T_{1}}^{r}|t-r|^{1-2b}dtdr\nonumber\\
&&=C(T_{2}-T_{1})^{3-2b}.\label{2.027}
\end{eqnarray}
For $I_{3}$, since $\left|\mathscr{F}_{x}f(\xi,t)\right|\leq C$, by applying Fubini's theorem, we have
\begin{eqnarray}
&&I_{3}=2\int_{T_{2}}^{\infty}\int_{T_{1}}^{T_{2}}\frac{|\mathscr{F}_{x}f(\xi,t)|^{2}}{|t-r|^{1+2b}}dtdr\nonumber\\
&&\leq C\int_{T_{2}}^{\infty}\int_{T_{1}}^{T_{2}}\frac{1}{|t-r|^{1+2b}}dtdr\nonumber\\
&&\leq C\int_{T_{1}}^{T_{2}}\int_{T_{2}}^{\infty}\frac{1}{|t-r|^{1+2b}}drdt\nonumber\\
&&\leq C\int_{T_{1}}^{T_{2}}(T_{2}-t)^{-2b}dt=C(T_{2}-T_{1})^{1-2b}.\label{2.028}
\end{eqnarray}
Combining \eqref{2.026}, \eqref{2.027} and \eqref{2.028}, we obtain
\begin{eqnarray}
&&L_{1}\leq C\left((T_{2}-T_{1})^{1-2b}+(T_{2}-T_{1})^{3-2b}\right).\label{2.029}
\end{eqnarray}
From \eqref{2.023}, we have
\begin{eqnarray}
&&L_{2}=\left\|\chi_{[T_{1},T_{2}]}(\cdot)\mathscr{F}_{x}f(\xi,t)\right\|_{L^{2}}^{2}\leq C(T_{2}-T_{1}).\label{2.030}
\end{eqnarray}
By using \eqref{2.023}, \eqref{2.024} and \eqref{2.029}, \eqref{2.030}, we have
\begin{eqnarray}
\left\|\chi_{[T_{1},T_{2}]}(\cdot)\mathscr{F}_{x}f(\xi,t)\right\|_{H^{b}}\leq C_{1}\left((T_{2}-T_{1})^{\frac{1}{2}}
+(T_{2}-T_{1})^{\frac{1}{2}-b}+(T_{2}-T_{1})^{\frac{3}{2}-b}\right).\label{2.031}
\end{eqnarray}

The proof of Lemma 2.7 is completed.

\begin{Lemma}\label{lem2.8}
Let $s\in\R$, $T_{1}, T_{2}\in [0,\infty)$ and $0\leq b<\frac{1}{2}$,
and suppose that  $\chi_{[T_{1},T_{2}]}(t)u\in X_{s,b}(\R^{2})$.
Then, for $0<\epsilon\leq 1$, there exists  a rapidly decreasing function $f$ such that
\begin{eqnarray}
&&\left\|J_{x}^{s}J_{t}^{b}\chi_{[T_{1},T_{2}]}(\cdot)S(-t)u\right\|_{L_{xt}^{2}}< C\left(\left\|\chi_{[T_{1},T_{2}]}(\cdot)J_{t}^{-b}f\right\|_{L_{x}^{2}(H_{t}^{b})}+\epsilon\right),\label{2.032}
\end{eqnarray}
where $S(t)=U(t)$ or $S(t)=V(t)$, and $C$ is a constant depending only on $b$.
\end{Lemma}
\noindent{\bf Proof.}
Since $\chi_{[T_{1},T_{2}]}(t)u\in X_{s,b}(\R^{2})$, we have $J_{x}^{s}J_{t}^{b}\chi_{[T_{1},T_{2}]}(t)S(-t)u\in L_{xt}^{2}(\R^{2})$.
Then, by the density theorem, for any for any $0<\epsilon\leq 1$,
there exists a rapidly decreasing function $f(x,t)$, and a function $g\in L_{xt}^{2}(\R^{2})$ such that
\begin{eqnarray}
&&J_{x}^{s}J_{t}^{b}\chi_{[T_{1},T_{2}]}(t)S(-t)u(x,t)=f(x,t)+g(x,t),\label{2.033}
\end{eqnarray}
where
\begin{eqnarray}
&&\|g\|_{L_{xt}^{2}}<\epsilon.\label{2.034}
\end{eqnarray}
From \eqref{2.033}, we have
\begin{eqnarray}
&&J_{x}^{s}J_{t}^{b}\chi_{[T_{1},T_{2}]}(t)S(-t)u(x,t)
=J_{t}^{b}\chi_{[T_{1},T_{2}]}(t)S(-t)J_{x}^{s}u(x,t)=f(x,t)+g(x,t),\nonumber\\
&&\chi_{[T_{1},T_{2}]}(t)S(-t)J_{x}^{s}u(x,t)=J_{t}^{-b}f(x,t)+J_{t}^{-b}g(x,t),\nonumber\\
&&\chi_{[T_{1},T_{2}]}(t)S(-t)J_{x}^{s}u(x,t)
=\chi_{[T_{1},T_{2}]}(t)J_{t}^{-b}f(x,t)+\chi_{[T_{1},T_{2}]}(t)J_{t}^{-b}g(x,t),\nonumber\\
&&J_{t}^{b}\chi_{[T_{1},T_{2}]}(t)S(-t)J_{x}^{s}u(x,t)\nonumber\\
&&=J_{t}^{b}\chi_{[T_{1},T_{2}]}(t)J_{t}^{-b}f(x,t)+J_{t}^{b}\chi_{[T_{1},T_{2}]}(t)J_{t}^{-b}g(x,t)
.\label{2.035}
\end{eqnarray}
By using \eqref{2.034}-\eqref{2.035} and $0\leq b<\frac{1}{2}$, we have
\begin{eqnarray}
&&\left\|J_{x}^{s}J_{t}^{b}\chi_{[T_{1},T_{2}]}(\cdot)S(-t)u\right\|_{L_{xt}^{2}}
\leq \left\|J_{t}^{b}\chi_{[T_{1},T_{2}]}(\cdot)J_{t}^{-b}f\right\|_{L_{xt}^{2}}
+\left\|J_{t}^{b}\chi_{[T_{1},T_{2}]}(\cdot)J_{t}^{-b}g\right\|_{L_{xt}^{2}}\nonumber\\
&&=\left\|\chi_{[T_{1},T_{2}]}(\cdot)J_{t}^{-b}f\right\|_{L_{x}^{2}(H_{t}^{b})}+\left\|\chi_{[T_{1},T_{2}]}(\cdot)J_{t}^{-b}g\right\|
_{L_{x}^{2}(H_{t}^{b})}\nonumber\\
&&\leq C\left(\left\|\chi_{[T_{1},T_{2}]}(\cdot)J_{t}^{-b}f\right\|_{L_{x}^{2}(H_{t}^{b})}
+\left\|J_{t}^{-b}g\right\|_{L_{x}^{2}(H_{t}^{b})}\right)\nonumber\\
&&\leq C\left(\left\|\chi_{[T_{1},T_{2}]}(\cdot)J_{t}^{-b}f\right\|_{L_{x}^{2}(H_{t}^{b})}
+\left\|g\right\|_{L_{xt}^{2}}\right)\nonumber\\
&&< C\left(\left\|\chi_{[T_{1},T_{2}]}
(\cdot)J_{t}^{-b}f\right\|_{L_{x}^{2}(H_{t}^{b})}+\epsilon\right).\label{2.036}
\end{eqnarray}
Note that the third inequality from the end in \eqref{2.036} requires $0\leq b<\frac{1}{2}$ to hold.

The proof of Lemma 2.8 is completed.

\begin{Lemma}\label{lem2.9}
Assume that $\chi_{[T_{1},T_{2}]}(t)u\in X_{s,b}(\R^{2})$ with $s\in\R$, $0\leq b<\frac{1}{2}$. Then for $T_{1}, T_{2}\in(0,1)$, we have
\begin{eqnarray}
&&\lim_{T_{2}\rightarrow T_{1}^{+}}\left\|\chi_{[T_{1},T_{2}]}(\cdot)u\right\|_{X_{s,b}}=0.\label{2.037}
\end{eqnarray}
\end{Lemma}
\noindent{\bf Proof.}
Notice that
\begin{eqnarray}
&&\left\|\chi_{[T_{1},T_{2}]}(\cdot)u\right\|_{X_{s,b}}
=\left\|J_{x}^{s}J_{t}^{b}\chi_{[T_{1},T_{2}]}(\cdot)S(-t)u\right\|_{L_{xt}^{2}},\label{2.038}
\end{eqnarray}
where $S(t)=U(t)$ or $S(t)=V(t)$.

\noindent By using Lemma 2.8, we have
\begin{eqnarray}
&&\left\|J_{x}^{s}J_{t}^{b}\chi_{[T_{1},T_{2}]}(\cdot)S(-t)u\right\|_{L_{xt}^{2}}
< C\left(\left\|\chi_{[T_{1},T_{2}]}(\cdot)J_{t}^{-b}f\right\|_{L_{x}^{2}(H_{t}^{b})}+\epsilon\right).\label{2.039}
\end{eqnarray}
Since $f(x,t)$ is a rapidly decreasing function, we have
\begin{eqnarray}
&&|\mathscr{F}_{x}J_{t}^{-b}f(\xi,t)|=\left|\int_{\SR}e^{it\tau}\langle\tau\rangle^{-b}\mathscr{F}_{xt}f(\xi,\tau)d\tau\right|\leq C\int_{\SR}|\mathscr{F}_{xt}f(\xi,\tau)|d\tau\nonumber\\
&&\leq C\int_{\SR}(1+|\tau|)^{-4}d\tau\leq C,\label{2.040}
\end{eqnarray}
and
\begin{eqnarray}
&&\left\|\chi_{[T_{1},T_{2}]}(\cdot)J_{t}^{-b}f(x,t)\right\|_{L_{x}^{2}(H_{t}^{b})}
\leq C\left\|J_{t}^{-b}f(x,t)\right\|_{L_{x}^{2}(H_{t}^{b})}\nonumber\\
&&\leq C\|f(x,t)\|_{L_{xt}^{2}}\leq C.\label{2.041}
\end{eqnarray}
By using \eqref{2.040} and Lemma 2.7, we have
\begin{eqnarray}
\left\|\chi_{[T_{1},T_{2}]}(\cdot)J_{t}^{-b}\mathscr{F}_{x}f(\xi,t)\right\|_{H_{t}^{b}}\leq C_{1}\left((T_{2}-T_{1})^{\frac{1}{2}}+(T_{2}-T_{1})^{\frac{1}{2}-b}+(T_{2}-T_{1})^{\frac{3}{2}-b}\right).\label{2.042}
\end{eqnarray}
On the one hand, it follows from \eqref{2.041} that there exists a constant  $M>0$ such that
\begin{eqnarray}
&&\left(\int_{|\xi|\geq M}\left\|\chi_{[T_{1},T_{2}]}(\cdot)J_{t}^{-b}
\mathscr{F}_{x}f(\xi,t)\right\|_{H_{t}^{b}}^{2}d\xi\right)^{\frac{1}{2}}<\epsilon.\label{2.043}
\end{eqnarray}
On the other hand, by using \eqref{2.042}, we have
\begin{eqnarray}
&&\left(\int_{|\xi|\leq M}\left\|\chi_{[T_{1},T_{2}]}(\cdot)J_{t}^{-b}
\mathscr{F}_{x}f(\xi,t)\right\|_{H_{t}^{b}}^{2}d\xi\right)^{\frac{1}{2}}\nonumber\\
&&\leq C_{1}(M)\left((T_{2}-T_{1})^{\frac{1}{2}}+(T_{2}-T_{1})^{\frac{1}{2}-b}+(T_{2}-T_{1})^{\frac{3}{2}-b}\right).\label{2.044}
\end{eqnarray}
From \eqref{2.039}, \eqref{2.043} and \eqref{2.044}, we have
\begin{eqnarray}
&&\left\|\chi_{[T_{1},T_{2}]}(\cdot)u\right\|_{X_{s,b}}
=\left\|J_{x}^{s}J_{t}^{b}\chi_{[T_{1},T_{2}]}(\cdot)S(-t)u\right\|_{L_{xt}^{2}}\nonumber\\
&&< C(M)\left((T_{2}-T_{1})^{\frac{1}{2}}+(T_{2}-T_{1})
^{\frac{1}{2}-b}+(T_{2}-T_{1})^{\frac{3}{2}-b}+2\epsilon\right).\label{2.045}
\end{eqnarray}
By using \eqref{2.045}, we have
\begin{eqnarray}
&&\lim_{T_{2}\rightarrow T_{1}^{+}}\left\|\chi_{[T_{1},T_{2}]}(\cdot)u\right\|_{X_{s,b}}=0.\label{2.046}
\end{eqnarray}

The proof of Lemma 2.9 is completed.

\bigskip

\bigskip

\section{Estimates related to the stochastic convolution}

\setcounter{equation}{0}

\setcounter{Theorem}{0}

\setcounter{Lemma}{0}

\setcounter{Proposition}{0}

\setcounter{section}{3}
In this section, we present some lemmas related to the stochastic convolution, which will be used in the subsequent proofs.

\begin{Lemma}\label{lem3.1}
Let $s\in\R$, $0<b<\frac{1}{2}$ and $\Phi_{1}\in L_{2}^{0,s}$, we have
\begin{eqnarray*}
&&\psi(t) \int_{0}^{t}U(t-s)\Phi_{1}dW(s)\in L^{2}(\Omega,X_{s,b,\varphi}),
\end{eqnarray*}
and
\begin{eqnarray*}
&&\mathbb{E}\left(\left\|\psi(t) \int_{0}^{t}U(t-s)\Phi_{1}dW(s)\right\|_{X_{s,b,\varphi}}^{2}\right)
\leq M(b,\psi)\|\Phi_{1}\|_{L_{2}^{0,s}}^{2},
\end{eqnarray*}
where $M(b, \psi)$ is a constant depending only on $b$,
$\|\psi\|_{H_{t}^{b}}$ , $\||t|^{\frac{1}{2}}\psi\|_{L^{2}}$, $\||t|^{\frac{1}{2}}\psi\|_{L^{\infty}}$.
\end{Lemma}

For Lemma 3.1, we refer to \cite[ Proposition 2.1]{BDT1999}.

\begin{Lemma}\label{lem3.2}
Let $s\in\R$, $0<b<\frac{1}{2}$ and $\Phi_{2}\in L_{2}^{0,s}$, we have
\begin{eqnarray*}
&&\psi(t)\int_{0}^{t}V(t-s)\Phi_{2}dW(s)\in L^{2}(\Omega,X_{s,b,\phi}),
\end{eqnarray*}
and
\begin{eqnarray*}
&&\mathbb{E}\left(\left\|\psi(t) \int_{0}^{t}V(t-s)\Phi_{2}dW(s)\right\|_{X_{s,b,\phi}}^{2}\right)
\leq M(b,\psi)\|\Phi_{2}\|_{L_{2}^{0,s}}^{2},
\end{eqnarray*}
where $M(b, \psi)$ is a constant depending only on $b$,
$\|\psi\|_{H_{t}^{b}}$ , $\||t|^{\frac{1}{2}}\psi\|_{L^{2}}$, $\||t|^{\frac{1}{2}}\psi\|_{L^{\infty}}$.
\end{Lemma}

For Lemma 3.2, we refer to \cite[Proposition 1]{WG2010}.

\begin{Lemma}(Integration by parts)\label{lem3.3}
Suppose that $f(s,\omega):=f(s)$ depends only on $s$, that $f\in C([0,t])$,
and that $f$ is of bounded variation on $[0,t]$. Then, we have
\begin{eqnarray*}
\int_{0}^{t}f(s)dB(s)=f(t)B(t)-\int_{0}^{t}B(s)df(s)го
\end{eqnarray*}
\end{Lemma}

For Lemma 3.3, we refer to \cite[Theorem 4.1.5]{O1998}.

\begin{Lemma}\label{lem3.4}
Let $s\in\R$, $\Phi_{1}\in L^{0,s}_{2}$, and
\begin{eqnarray}
&&w_{1}^{N}=\sum_{i=N+1}^{\infty}\int_{0}^{t}U(t-s)\Phi_{1} e_{i}d\beta_{i}(s),\label{3.01}\\
&&P^{N}w_{1N}=\sum_{i=1}^{N}\int_{0}^{t}U(t-s)P^{N}\Phi_{1} e_{i}d\beta_{i}(s).\label{3.02}
\end{eqnarray}
Then, we have
\begin{eqnarray}
&&\mathbb{E}\left(\sup_{t\in[0,1]}\|w_{1}^{N}(t,\cdot)\|_{H^{s}}^{2}\right)
\leq C\sum_{i=N+1}^{\infty}\|\Phi_{1} e_{i}\|_{H^{s}}^{2},\label{3.03}\\
&&\mathbb{E}\left(\sup_{t\in[0,1]}\|P^{N}w_{1N}(t,\cdot)\|_{H^{s}}^{2}\right)
\leq C\sum_{i=1}^{N}\|P^{N}\Phi_{1} e_{i}\|_{H^{s}}^{2}.\label{3.04}
\end{eqnarray}

\end{Lemma}

\noindent{\bf Proof.}
Inspired by \cite[Theorem 6.10]{PZ2014}, we prove Lemma 3.4.

\noindent We denote
\begin{eqnarray}
&&w^{N}:=\sum_{i=N+1}^{\infty}\int_{0}^{t}U(-s)\Phi_{1} e_{i}d\beta_{i}(s),\label{3.05}\\
&&P^{N}w_{N}:=\sum_{i=1}^{N}\int_{0}^{t}U(-s)P^{N}\Phi_{1} e_{i}d\beta_{i}(s).\label{3.06}
\end{eqnarray}
From \eqref{3.05} and \eqref{3.06}, we have
\begin{eqnarray}
&&\|w_{1}^{N}(t,\cdot)\|_{H^{s}}=\|U(t)w^{N}(t,\cdot)\|_{H^{s}}=\|w^{N}(t,\cdot)\|_{H^{s}},\label{3.07}\\
&&\|P^{N}w_{1N}(t,\cdot)\|_{H^{s}}=\|U(t)P^{N}w_{N}(t,\cdot)\|_{H^{s}}=\|P^{N}w_{N}(t,\cdot)\|_{H^{s}}.\label{3.08}
\end{eqnarray}
Next, we prove \eqref{3.03} and \eqref{3.04}, respectively. For \eqref{3.03}, by using \eqref{3.05}, we have
\begin{eqnarray}
&&d J_{x}^{s}w^{N}(t)=\sum_{i=N+1}^{\infty}U(-t)J_{x}^{s}\Phi_{1} e_{i}d\beta_{i}(t).\label{3.09}
\end{eqnarray}
We denote $F:L^{2}\mapsto\R$, and for $X\in L^{2}$,
\begin{eqnarray}
&&F(X)=\|X\|_{L^{2}}^{2}.\label{3.010}
\end{eqnarray}
From \eqref{3.010}, we have
\begin{eqnarray}
&&F^{\prime}(X)=2X,\,F^{\prime\prime}(X)=2.\label{3.011}
\end{eqnarray}
By using \eqref{3.09}-\eqref{3.011} and It\^{o} formula, we have
\begin{eqnarray}
&&\|w_{1}^{N}(t,\cdot)\|_{H^{s}}^{2}=\|w^{N}(t,\cdot)\|_{H^{s}}^{2}=\|J_{x}^{s}w^{N}(t,\cdot)\|_{L_{x}^{2}}^{2}\nonumber\\
&&=2\sum_{i=N+1}^{\infty}\int_{0}^{t}\langle J_{x}^{s}w^{N}, J_{x}^{s}U(-s)\Phi_{1} e_{i}\rangle_{L_{x}^{2}} d\beta_{i}(s)+\int_{0}^{t}\sum_{i=N+1}^{\infty}\|J_{x}^{s}\Phi_{1} e_{i}\|_{L_{x}}^{2}ds\nonumber\\
&&=2\sum_{i=N+1}^{\infty}\int_{0}^{t}\langle J_{x}^{s}w_{1}^{N}, J_{x}^{s}\Phi_{1} e_{i}\rangle_{L_{x}^{2}} d\beta_{i}(s)
+\int_{0}^{t}\sum_{i=N+1}^{\infty}\|\Phi_{1} e_{i}\|_{H^{s}}^{2}ds.\label{3.012}
\end{eqnarray}
From \eqref{3.012}, we have
\begin{eqnarray}
&&\|w_{1}^{N}(t,\cdot)\|_{H^{s}}^{2}=2\sum_{i=N+1}^{\infty}\int_{0}^{t}\langle J_{x}^{s}w_{1}^{N}, J_{x}^{s}\Phi_{1} e_{i}\rangle_{L_{x}^{2}} d\beta_{i}(s)+\int_{0}^{t}\sum_{i=N+1}^{\infty}\|\Phi_{1} e_{i}\|_{H^{s}}^{2}ds\nonumber\\
&&=H_{1}(t)+H_{2}(t),\label{3.013}
\end{eqnarray}
where
\begin{eqnarray*}
&&H_{1}(t)=2\sum_{i=N+1}^{\infty}\int_{0}^{t}\langle J_{x}^{s}w_{1}^{N}, J_{x}^{s}\Phi_{1} e_{i}\rangle_{L_{x}^{2}} d\beta_{i}(s),\\
&&H_{2}(t)=\int_{0}^{t}\sum_{i=N+1}^{\infty}\|\Phi_{1} e_{i}\|_{H^{s}}^{2}ds.
\end{eqnarray*}
For $H_{1}(t)$, by using Burkholder-Davis-Gundy (BDG) inequality, H\"{o}lder inequality and Young's inequality, we have
\begin{eqnarray}
&&\mathbb{E}\left(\sup_{t\in[0,1]} H_{1}(t)\right)\leq 6\mathbb{E}\left[\left(\int_{0}^{1}
\sum_{i=N+1}^{\infty}\langle J_{x}^{s}w_{1}^{N}, J_{x}^{s}\Phi_{1} e_{i}\rangle_{L_{x}^{2}}^{2}ds\right)^{\frac{1}{2}}\right]\nonumber\\
&&\leq C\mathbb{E}\left[\left(\int_{0}^{1}\sum_{i=N+1}^{\infty}\|J_{x}^{s}w_{1}^{N}\|_{L_{x}^{2}}^{2}\|J_{x}^{s}\Phi_{1} e_{i}\|_{L_{x}^{2}}^{2}ds\right)^{\frac{1}{2}}\right]\nonumber\\
&&=C\left(\sum_{i=N+1}^{\infty}\|\Phi_{1} e_{i}\|_{H^{s}}^{2}\right)^{\frac{1}{2}}
\mathbb{E}\left[\left(\int_{0}^{1}\|J_{x}^{s}w_{1}^{N}\|_{L_{x}^{2}}^{2}ds\right)^{\frac{1}{2}}\right]\nonumber\\
&&\leq C\left(\sum_{i=N+1}^{\infty}\|\Phi_{1} e_{i}\|_{H^{s}}^{2}\right)^{\frac{1}{2}}
\mathbb{E}\left(\sup_{t\in[0,1]}\|w_{1}^{N}(t,\cdot)\|_{H^{s}}\right)\nonumber\\
&&\leq C\left(\sum_{i=N+1}^{\infty}\|\Phi_{1} e_{i}\|_{H^{s}}^{2}\right)^{\frac{1}{2}}
\left[\mathbb{E}\left(\sup_{t\in[0,1]}\|w_{1}^{N}(t,\cdot)\|_{H^{s}}^{2}\right)\right]^{\frac{1}{2}}\nonumber\\
&&\leq \frac{1}{2}\mathbb{E}\left(\sup_{t\in[0,1]}\|w_{1}^{N}(t,\cdot)\|_{H^{s}}^{2}\right)
+C_{1}\sum_{i=N+1}^{\infty}\|\Phi_{1} e_{i}\|_{H^{s}}^{2}.\label{3.014}
\end{eqnarray}
For $H_{2}(t)$, we have
\begin{eqnarray}
&&\mathbb{E}\left(\sup_{t\in[0,1]} H_{2}(t)\right)\leq C\sum_{i=N+1}^{\infty}\|\Phi_{1} e_{i}\|_{H^{s}}^{2}.\label{3.015}
\end{eqnarray}
Combining \eqref{3.013}, \eqref{3.014} with \eqref{3.015}, we have
\begin{eqnarray}
&&\mathbb{E}\left(\sup_{t\in[0,1]}\|w_{1}^{N}(t,\cdot)\|_{H^{s}}^{2}\right)
\leq \frac{1}{2}\mathbb{E}\left(\sup_{t\in[0,1]}\|w_{1}^{N}(t,\cdot)\|_{H^{s}}^{2}\right)
+C_{2}\sum_{i=N+1}^{\infty}\|\Phi_{1} e_{i}\|_{H^{s}}^{2}.\label{3.016}
\end{eqnarray}
From \eqref{3.016}, we have
\begin{eqnarray}
&&\mathbb{E}\left(\sup_{t\in[0,1]}\|w_{1}^{N}(t,\cdot)\|_{H^{s}}^{2}\right)
\leq 2C_{2}\sum_{i=N+1}^{\infty}\|\Phi_{1} e_{i}\|_{H^{s}}^{2}.\label{3.017}
\end{eqnarray}
It follows from \eqref{3.017} that \eqref{3.03} is valid.

\noindent For \eqref{3.04}, by using \eqref{3.06}, we have
\begin{eqnarray}
&&d J_{x}^{s}P^{N}w_{N}(t)=\sum_{i=1}^{N}U(-t)J_{x}^{s}P^{N}\Phi_{1} e_{i}d\beta_{i}(t).\label{3.018}
\end{eqnarray}
By using \eqref{3.018} and a similar argument as in the proof of \eqref{3.03}, we obtain \eqref{3.04}.

The proof of Lemma 3.4 is completed.

\begin{Lemma}\label{lem3.5}
Let $s\in\R$,
\begin{eqnarray*}
&&\hspace{-0.5cm}A=\left\{\omega: \lim_{k\rightarrow\infty}\sup_{t\in[0,T_{\omega}]}\|f_{n_{k}}(t,\omega)-f(\omega)\|_{H^{s}}=0,
{\rm and}\, \forall n\in\Z^{+},\,\lim_{t\rightarrow t_{0}}\|f_{n}(t,\omega)-f_{n}(t_{0},\omega)\|_{H^{s}}=0\right\},\\
&&\hspace{-0.5cm}B=\left\{\omega: \forall t,t_{0}\in[0,T_{\omega}]\lim_{t\rightarrow t_{0}}\|f(t,\omega)-f(t_{0},\omega)\|_{H^{s}}=0\right\},
\end{eqnarray*}
and
\begin{eqnarray}
&&\mathbb{P}(A)=1.\label{3.019}
\end{eqnarray}
Then, we have
\begin{eqnarray}
&&\mathbb{P}\left(B\right)=1.\label{3.020}
\end{eqnarray}
\end{Lemma}

\noindent{\bf Proof.}
We claim that
\begin{eqnarray}
&&A\subset B.\label{3.021}
\end{eqnarray}
Now, we prove \eqref{3.021}. By using a pointwise argument, for any $\omega\in A$, we have
\begin{eqnarray}
&&\lim_{k\rightarrow\infty}\sup_{t\in[0,T_{\omega}]}\|f_{n_{k}}(t,\omega)-f(t,\omega)\|_{H^{s}}=0,\label{3.022}\\
&&\lim_{t\rightarrow t_{0}}\|f_{n}(t,\omega)-f_{n}(t_{0},\omega)\|_{H^{s}}=0,\, \forall n\in\Z^{+}.\label{3.023}
\end{eqnarray}
From \eqref{3.022}, we know that for all $\epsilon>0$, there exists $K$, such that for all $k\geq K$, we have
\begin{eqnarray}
&&\sup_{t\in [0,T_{\omega}]}\left\|f(t,\omega)-f_{n_{k}}(t,\omega)\right\|_{H^{s}}<\epsilon.\label{3.024}
\end{eqnarray}
Now, we prove that $\omega\in B$. By using triangle inequality, for any $t,t_{0}\in[0, T_{\omega}]$, we have
\begin{eqnarray}
&&\|f(t,\omega)-f(t_{0},\omega)\|_{H^{s}}\leq \|f(t,\omega)-f_{n_{K+1}}(t,\omega)\|_{H^{s}}
+\|f_{n_{K+1}}(t,\omega)-f_{n_{K+1}}(t_{0},\omega)\|_{H^{s}}\nonumber\\
&&\quad+\|f_{n_{K+1}}(t_{0},\omega)-f(t_{0},\omega)\|_{H^{s}}\nonumber\\
&&\leq 2\sup_{t\in[0,T_{\omega}]}\|f(t,\omega)-f_{n_{K+1}}(t,\omega)\|_{H^{s}}
+\|f_{n_{K+1}}(t,\omega)-f_{n_{K+1}}(t_{0},\omega)\|_{H^{s}}.\label{3.025}
\end{eqnarray}
By using \eqref{3.023}-\eqref{3.025}, we have
\begin{eqnarray}
&&\lim_{t\rightarrow t_{0}}\|f(t,\omega)-f(t_{0},\omega)\|_{H^{s}}\leq 2\lim_{t\rightarrow t_{0}}\sup_{t\in[0,T_{\omega}]}\|f(t,\omega)-f_{n_{K+1}}(t,\omega)\|_{H^{s}}\nonumber\\
&&\quad+\lim_{t\rightarrow t_{0}}\|f_{n_{K+1}}(t,\omega)-f_{n_{K+1}}(t_{0},\omega)\|_{H^{s}}\nonumber\\
&&<2\epsilon.\label{3.026}
\end{eqnarray}
It follows from \eqref{3.026} that
\begin{eqnarray}
&&\lim_{t\rightarrow t_{0}}\|f(t,\omega)-f(t_{0},\omega)\|_{H^{s}}=0.\label{3.027}
\end{eqnarray}
By using \eqref{3.027}, we have that \eqref{3.021} is valid.
Combining \eqref{3.019} with \eqref{3.021}, we have that \eqref{3.020} is valid.

The proof of Lemma 3.5 is completed.

\begin{Lemma}\label{lem3.6}
Let $s\in\R$, $\Phi_{1}\in L_{2}^{0,s}$. Then, we have
\begin{eqnarray}
&&\mathbb{P}\left(\left\{\omega: \forall t_{0}, t\in [0,T_{\omega}], T_{\omega}\in (0,1), \lim_{t\rightarrow t_{0}}\left\|w_{1}(t,\omega,\cdot)-w_{1}(t_{0},\omega,\cdot)\right\|_{H^{s}}=0\right\}\right)\nonumber\\
&&\quad=1,\label{3.028}
\end{eqnarray}
where
\begin{eqnarray*}
&&w_{1}=\int_{0}^{t}U(t-s)\Phi_{1}dW(s).
\end{eqnarray*}

\end{Lemma}

\noindent{\bf Proof.}
Since $(e_{i})_{i=1}^{\infty}$ is an orthonormal basis of $L^{2}$,
$\Phi_{1}: L^{2}\rightarrow H^{s}$ is a Hilbert-Schmidt operator, it follows that
\begin{eqnarray}
&&\|\Phi_{1} e_{i}\|_{H^{s}}<\infty,\label{3.029}\\
&&\sum_{i=1}^{\infty}\|\Phi_{1} e_{i}\|_{H^{s}}^{2}<\infty.\label{3.030}
\end{eqnarray}
From \eqref{3.030}, for any $\epsilon>0$, there exists $N_{1}\geq 1$ with $N_{1}\in\mathbb{N}^{+}$, such that
\begin{eqnarray}
&&\sum_{i=N_{1}+1}^{\infty}\|\Phi_{1} e_{i}\|_{H^{s}}^{2}<\frac{\epsilon}{2}.\label{3.031}
\end{eqnarray}
It follows from \eqref{3.029} that for the above $\epsilon>0$, there exists $N_{2}\in \mathbb{N}^{+}$, such that
\begin{eqnarray}
&&\sum_{i=1}^{N_{1}}\|P^{N_{2}}\Phi_{1} e_{i}\|_{H^{s}}^{2}< \frac{\epsilon}{2}.\label{3.032}
\end{eqnarray}
We consider the following cases:

\noindent {\bf Case 1.} When $N_{1}=N_{2}$, from \eqref{3.031} and \eqref{3.032}, we have
\begin{eqnarray}
&&\sum_{i=N_{1}+1}^{\infty}\|\Phi_{1} e_{i}\|_{H^{s}}^{2}+\sum_{i=1}^{N_{1}}\|P^{N_{1}}\Phi_{1} e_{i}\|_{H^{s}}^{2}< \epsilon
.\label{3.033}
\end{eqnarray}
{\bf Case 2.} When $N_{1}>N_{2}$, from \eqref{3.031} and \eqref{3.032}, we have
\begin{eqnarray}
&&\sum_{i=N_{1}+1}^{\infty}\|\Phi_{1} e_{i}\|_{H^{s}}^{2}+\sum_{i=1}^{N_{1}}\|P^{N_{1}}\Phi_{1} e_{i}\|_{H^{s}}^{2}\nonumber\\
&&\leq \sum_{i=N_{1}+1}^{\infty}\|\Phi_{1} e_{i}\|_{H^{s}}^{2}+\sum_{i=1}^{N_{1}}\|P^{N_{2}}\Phi_{1} e_{i}\|_{H^{s}}^{2}< \epsilon
.\label{3.034}
\end{eqnarray}
{\bf Case 3.} When $N_{1}<N_{2}$, from \eqref{3.031} and \eqref{3.032}, we have
\begin{eqnarray}
&&\sum_{i=N_{2}+1}^{\infty}\|\Phi_{1} e_{i}\|_{H^{s}}^{2}+\sum_{i=1}^{N_{2}}\|P^{N_{2}}\Phi_{1} e_{i}\|_{H^{s}}^{2}\nonumber\\
&&\leq \sum_{i=N_{2}+1}^{\infty}\|\Phi_{1} e_{i}\|_{H^{s}}^{2}+\sum_{i=1}^{N_{1}}\|P^{N_{2}}\Phi_{1} e_{i}\|_{H^{s}}^{2}
+\sum_{i=N_{1}+1}^{N_{2}}\|\Phi_{1} e_{i}\|_{H^{s}}^{2}\nonumber\\
&&= \frac{\epsilon}{2}+\sum_{i=N_{1}+1}^{\infty}\|\Phi_{1} e_{i}\|_{H^{s}}^{2}<\epsilon
.\label{3.035}
\end{eqnarray}
Let $N=\max\{N_{1}, N_{2}\}$. Combining \eqref{3.033}, \eqref{3.034} with \eqref{3.035}, we obtain
\begin{eqnarray}
&&\sum_{i=N+1}^{\infty}\|\Phi_{i} e_{i}\|_{H^{s}}^{2}
+\sum_{i=1}^{N}\|P^{N}\Phi_{1} e_{i}\|_{H^{s}}^{2}< \epsilon.\label{3.036}
\end{eqnarray}
For the above $N$ and $n\geq N$, $n\in\mathbf{N}^{+}$,
we have
\begin{eqnarray}
&&\sum_{i=n+1}^{\infty}\|\Phi_{i} e_{i}\|_{H^{s}}^{2}+\sum_{i=1}^{n}\|P^{n}\Phi_{1} e_{i}\|_{H^{s}}^{2}\nonumber\\
&&\leq \sum_{i=n+1}^{\infty}\|\Phi_{i} e_{i}\|_{H^{s}}^{2}
+\sum_{i=1}^{N}\|P^{n}\Phi_{1} e_{i}\|_{H^{s}}^{2}
+\sum_{i=N+1}^{n}\|\Phi_{1} e_{i}\|_{H^{s}}^{2}\nonumber\\
&&\leq\sum_{i=N+1}^{\infty}\|\Phi_{i} e_{i}\|_{H^{s}}^{2}
+\sum_{i=1}^{N}\|P^{N}\Phi_{1} e_{i}\|_{H^{s}}^{2}< \epsilon,\label{3.037}
\end{eqnarray}
and we define
\begin{eqnarray}
&&f_{i}:=\int_{0}^{t}U(t-s)\Phi_{1} e_{i}d\beta_{j}(s),
\, P_{n}f_{i}=\int_{0}^{t}U(t-s)P_{n}\Phi_{1} e_{i}d\beta_{i}(s).\label{3.038}\\
&&F_{n}:=\sum_{i=1}^{n}P_{n}\int_{0}^{t}U(t-s)\Phi_{1} e_{i}d\beta_{i}(s)=\sum_{i=1}^{n}P_{n}f_{i}.\label{3.039}
\end{eqnarray}
Next, for each fixed $n\geq N$, $n\in\mathbf{N}^{+}$,
we establish that $\|F_{n}\|_{H^{s}}$ is continuous with respect to $t\in[0,1]$. We note that
\begin{eqnarray}
&&\left\|\sum_{i=1}^{n}P_{n}f_{i}(t)-\sum_{i=1}^{n}P_{n}f_{i}(t_{0})\right\|_{H^{s}}\leq n^{\frac{1}{2}}\left(\sum_{i=1}^{n}\left\|P_{n}f_{i}(t)-P_{n}f_{i}(t_{0})\right\|_{H^{s}}^{2}\right)^{\frac{1}{2}}\nonumber\\
&&=n^{\frac{1}{2}}\left(\sum_{i=1}^{n}\int_{|\xi|\leq n}\langle\xi\rangle^{2s}
\left|\int_{0}^{t}e^{i(t-s)\xi^{3}}\mathscr{F}_{x}\Phi_{1} e_{i}(\xi)
d\beta_{i}(s)-\int_{0}^{t_{0}}e^{i(t_{0}-s)\xi^{3}}\mathscr{F}_{x}\Phi_{1} e_{i}(\xi)
d\beta_{i}(s)\right|^{2}d\xi\right)^{\frac{1}{2}}\nonumber\\
&&=n^{\frac{1}{2}}\left(\sum_{i=1}^{n}\int_{|\xi|\leq n}\langle\xi\rangle^{2s}
\left|\mathscr{F}_{x}\Phi_{1} e_{i}(\xi)\right|^{2}\left|\int_{0}^{t}e^{i(t-s)\xi^{3}}
d\beta_{i}(s)-\int_{0}^{t_{0}}e^{i(t_{0}-s)\xi^{3}}
d\beta_{i}(s)\right|^{2}d\xi\right)^{\frac{1}{2}}.\label{3.040}
\end{eqnarray}
Applying Lemma 3.3 and the triangle inequality, we obtain
\begin{eqnarray}
&&\left|\int_{0}^{t}e^{i(t-s)\xi^{3}}
d\beta_{i}(s)-\int_{0}^{t_{0}}e^{i(t_{0}-s)\xi^{3}}
d\beta_{i}(s)\right|\nonumber\\
&&=\left|e^{it\xi^{3}}\left[e^{-it\xi^{3}}\beta_{i}(t)+i\xi^{3}\int_{0}^{t}\beta_{i}(s)e^{-is\xi^{3}}d\xi\right]
-e^{it_{0}\xi^{3}}\left[e^{-it_{0}\xi^{3}}\beta_{i}(t)
+i\xi^{3}\int_{0}^{t_{0}}\beta_{i}(s)e^{-is\xi^{3}}d\xi\right]\right|\nonumber\\
&&\leq \Big|\beta_{i}(t)-\beta_{i}(t_{0})\Big|\nonumber\\
&&\,+|\xi|^{3}\left|e^{it\xi^{3}}\int_{0}^{t}\beta_{i}(s)e^{-is\xi^{3}}ds-
e^{it\xi^{3}}\int_{0}^{t_{0}}\beta_{i}(s)e^{-is\xi^{3}}ds\right|\nonumber\\
&&\,+|\xi|^{3}\left|e^{it\xi^{3}}\int_{0}^{t_{0}}\beta_{i}(s)e^{-is\xi^{3}}ds-
e^{it_{0}\xi^{3}}\int_{0}^{t_{0}}\beta_{i}(s)e^{-is\xi^{3}}ds\right|\nonumber\\
&&=\Big|\beta_{i}(t)-\beta_{i}(t_{0})\Big|+|\xi|^{3}\left|\int_{t_{0}}^{t}\beta_{i}(s)e^{-is\xi^{3}}ds\right|
+|\xi|^{3}\left|e^{i(t-t_{0})\xi^{3}}-1\right|\left|\int_{0}^{t_{0}}\beta_{i}(s)e^{-is\xi^{3}}ds\right|\nonumber\\
&&\leq \Big|\beta_{i}(t)-\beta_{i}(t_{0})\Big|+C(\omega)(N+n)^{3}|t-t_{0}|+C(\omega)n^{6}|t-t_{0}|\nonumber\\
&&=\left|\beta_{i}(t)-\beta_{i}(t_{0})\right|+C(\omega)n^{3}(1+n^{3})|t-t_{0}|,\label{3.041}
\end{eqnarray}
where $C(\omega)=\sup\limits_{t\in[0,1]} |\beta_{i}(t)|$.
It follows from \eqref{3.040} and \eqref{3.041} that
\begin{eqnarray}
&&\left\|\sum_{i=1}^{n}P_{n}f_{i}(t)-\sum_{i=1}^{n}P_{n}f_{i}(t_{0})\right\|_{H^{s}}\nonumber\\
&&\leq Cn^{\frac{1}{2}}\left(\left|\beta_{i}(t)-\beta_{i}(t_{0})\right|+C(\omega)n^{3}(1+n^{3})|t-t_{0}|\right)\nonumber\\
&&\quad\times\left(\sum_{i=1}^{n}\int_{|\xi|\leq n}\langle\xi\rangle^{2s}
\left|\mathscr{F}_{x}\Phi_{1} e_{i}(\xi)\right|^{2}d\xi\right)^{\frac{1}{2}}\nonumber\\
&&\leq Cn^{\frac{1}{2}}\left(\left|\beta_{i}(t)-\beta_{i}(t_{0})\right|+C(\omega)n^{3}(1+n^{3})|t-t_{0}|\right)\nonumber\\
&&\quad\times\left(\sum_{i=1}^{\infty}\|\Phi_{1} e_{i}\|_{H^{s}}^{2}\right)^{\frac{1}{2}}.\label{3.042}
\end{eqnarray}
Since $\beta_{i}(t)$ is a Brownian motion, for any $\epsilon>0$, there exists  $\delta>0$, such that when $$|t-t_{0}|<\delta<\frac{\epsilon}{2Cn^{\frac{1}{2}}C(\omega)n^{3}(1+n^{3})
\left(\sum\limits_{i=1}^{\infty}\|\Phi_{1} e_{i}\|_{H^{s}}^{2}\right)^{\frac{1}{2}}},$$
we have
\begin{eqnarray}
&&\left|\beta_{i}(t)-\beta_{i}(t_{0})\right|<\frac{\epsilon}{2Cn^{\frac{1}{2}}\left(\sum\limits_{i=1}^{\infty}\|\Phi_{1} e_{i}\|_{H^{s}}^{2}\right)^{\frac{1}{2}}}.\label{3.043}
\end{eqnarray}
On the other hand, when
$$|t-t_{0}|<\frac{\epsilon}{2Cn^{\frac{1}{2}}C(\omega)n^{3}(1+n^{3})
\left(\sum\limits_{i=1}^{\infty}\|\Phi_{1} e_{i}\|_{H^{s}}^{2}\right)^{\frac{1}{2}}},$$
we have
\begin{eqnarray}
&&Cn^{\frac{1}{2}}\left(C(\omega)n^{3}(1+n^{3})|t-t_{0}|\right)\left(\sum\limits_{i=1}^{\infty}\|\Phi_{1} e_{i}\|_{H^{s}}^{2}\right)^{\frac{1}{2}}<\frac{\epsilon}{2}.\label{3.044}
\end{eqnarray}
Therefore, combining \eqref{3.042}, \eqref{3.043} with \eqref{3.044}, we have that for any $\epsilon>0$, there exists
$$0<\delta<\frac{\epsilon}{2Cn^{\frac{1}{2}}
C(\omega)n^{3}(1+n^{3})\left(\sum\limits_{i=1}^{\infty}\|\Phi_{1} e_{i}\|_{H^{s}}^{2}\right)^{\frac{1}{2}}},$$
such that when  $|t-t_{0}|<\delta$, we have
\begin{eqnarray}
&&\left\|\sum_{i=1}^{n}P_{n}f_{i}(t)-\sum_{i=1}^{n}P_{n}f_{i}(t_{0})\right\|_{H^{s}}<\frac{\epsilon}{2}+\frac{\epsilon}{2}=\epsilon.\label{3.045}
\end{eqnarray}
From \eqref{3.039} and \eqref{3.045}, we have
\begin{eqnarray}
\mathbb{P}\left(\left\{\omega: t, t_{0}\in [0, T_{\omega}]\subset [0,1], \lim_{t\rightarrow t_{0}}\left\|F_{n}(t,\omega,\cdot)-F_{n}(t_{0},\omega,\cdot)\right\|_{H^{s}}=0\right\}\right)=1.\label{3.046}
\end{eqnarray}

\noindent By using Lemma 3.4, \eqref{3.037} and Chebyshev's inequality, for any $\alpha>0$(fixed), we have
\begin{eqnarray}
&&\mathbb{P}\left(\left\{\omega:\sup_{t\in [0,1]}\|w_{1}-F_{n}\|_{H^{s}}>\alpha\right\}\right)\nonumber\\
&&\leq C\frac{\mathbb{E}\left(\sup\limits_{t\in[0,1]}\|w_{1}-F_{n}\|_{H^{s}}^{2}\right)}{\alpha^{2}}\nonumber\\
&&\leq\frac{C\left(\sum\limits_{i=n+1}^{\infty}\|\Phi_{1} e_{i}\|_{H^{s}}^{2}
+\sum\limits_{i=1}^{n}\|P^{n}\Phi_{1} e_{i}\|_{H^{s}}^{2}\right)}{\alpha^{2}}\nonumber\\
&&<\frac{C\epsilon}{\alpha^{2}}.\label{3.047}
\end{eqnarray}
From \eqref{3.047}, we have
\begin{eqnarray}
&&\lim_{n\rightarrow\infty}\mathbb{P}\left(\left\{\omega:\sup_{t\in [0,1]}\|w_{1}-F_{n}\|_{H^{s}}>\alpha\right\}\right)=0
.\label{3.048}
\end{eqnarray}
By using \eqref{3.048}, we have
\begin{eqnarray}
&&\lim_{n\rightarrow\infty}\mathbb{P}\left(\left\{\omega:\sup_{t\in [0,T_{\omega}]}\|w_{1}-F_{n}\|_{H^{s}}>\alpha\right\}\right)
=0.\label{3.049}
\end{eqnarray}
It follows from \eqref{3.049} that
\begin{eqnarray}
&&\sup_{t\in [0,T_{\omega}]}\|w_{1}-F_{n}\|_{H^{s}}\overset{\mathbb{P}}{\longrightarrow}0.\label{3.050}
\end{eqnarray}
From \eqref{3.050}, there exists a subsequence $F_{n_{k}}$ such that
\begin{eqnarray}
&&\mathbb{P}\left(\left\{\omega:\lim_{k\rightarrow\infty}\sup_{t\in [0,T_{\omega}]}\|w_{1}-F_{n_{k}}\|_{H^{s}}=0\right\}\right)=1.\label{3.051}
\end{eqnarray}
From \eqref{3.046} and \eqref{3.051}, we have
\begin{eqnarray}
&&\mathbb{P}\left(A\right)=1,\label{3.052}
\end{eqnarray}
where
\begin{eqnarray*}
&&A=\left\{\omega:\lim_{k\rightarrow\infty}\sup_{t\in [0,T_{\omega}]}\|w_{1}-F_{n_{k}}\|_{H^{s}}=0,\right.\\
&&\left.{\rm and}\, \forall n_{k}\in\Z^{+},\,\lim_{t\rightarrow t_{0}}
\|F_{n_{k}}(t,\omega,\cdot)-F_{n_{k}}(t_{0},\omega,\cdot)\|_{H^{s}}=0\right\}.
\end{eqnarray*}
By using \eqref{3.052} and Lemma 3.5, we have
\begin{eqnarray}
&&\mathbb{P}\left(\left\{\omega: \forall t_{0},t\in [0,T_{\omega}], T_{\omega}\in(0,1), \lim_{t\rightarrow t_{0}}\|w_{1}(t,\omega,\cdot)-w_{1}(t_{0},\omega,\cdot)\|_{H^{s}}=0\right\}\right)\nonumber\\
&&\quad=1.\label{3.053}
\end{eqnarray}

The proof of Lemma 3.6 is completed.

\begin{Lemma}\label{lem3.7}
Let $s\in\R$, $\Phi_{1}\in L_{2}^{0,s}$, $0<b<\frac{1}{2}$ and
\begin{eqnarray*}
&&G(T,t,x,\omega)=\chi_{[0,T]}\int_{0}^{t}U(t-s)\Phi_{1} dW(s).
\end{eqnarray*}
Then, for each fixed $T$,
\begin{eqnarray*}
&&\left\|G(T,\cdot,\cdot,\omega)\right\|_{X_{s,b,\varphi}}
\end{eqnarray*}
is $\mathcal{F}_{T}$ measurable.
\end{Lemma}

\noindent{\bf Proof.}
Since
\begin{eqnarray}
&&W(t)=\sum_{i=1}^{\infty}\beta_{i}(t)e_{i},\label{3.054}
\end{eqnarray}
it follows from \eqref{3.054} that
\begin{eqnarray}
&&G(T,t,x,\omega)=\chi_{[0,T]}\int_{0}^{t}U(t-s)\Phi_{1}dW(s)\nonumber\\
&&=\sum_{i=1}^{\infty}\chi_{[0,T]}\int_{0}^{t}U(t-s)\Phi_{1} e_{i}d\beta_{i}(s).\label{3.055}
\end{eqnarray}
Since $(e_{i})_{i=1}^{\infty}$ is an orthonormal basis of $L^{2}$,
$\Phi_{1}: L^{2}\rightarrow H^{s}$ is a Hilbert-Schmidt operator, it follows that
\begin{eqnarray}
&&\|\Phi_{1} e_{i}\|_{H^{s}}<\infty,\label{3.056}\\
&&\sum_{i=1}^{\infty}\|\Phi_{1} e_{i}\|_{H^{s}}^{2}<\infty.\label{3.057}
\end{eqnarray}
From \eqref{3.057}, for every $\epsilon>0$, there exists $M_{1}\geq 1$ with $M_{1}\in\mathbb{N}^{+}$, such that
\begin{eqnarray}
&&\sum_{i=M_{1}+1}^{\infty}\|\Phi_{1} e_{i}\|_{H^{s}}^{2}<\frac{\epsilon}{2}.\label{3.058}
\end{eqnarray}
It follows from \eqref{3.056} that for the above $\epsilon>0$, there exists $M_{2}\in\mathbb{N}^{+}$, such that
\begin{eqnarray}
&&\sum_{i=1}^{M_{1}}\|P^{M_{2}}\Phi_{1} e_{i}\|_{H^{s}}^{2}< \frac{\epsilon}{2}.\label{3.059}
\end{eqnarray}
We consider the following cases:

\noindent{\bf Case 1.} When $M_{1}=M_{2}$, from \eqref{3.058} and \eqref{3.059}, we have
\begin{eqnarray}
&&\sum_{i=M_{1}+1}^{\infty}\|\Phi_{1} e_{i}\|_{H^{s}}^{2}+\sum_{i=1}^{M_{1}}\|P^{M_{1}}\Phi_{1} e_{i}\|_{H^{s}}^{2}< \epsilon
.\label{3.060}
\end{eqnarray}
{\bf Case 2.} When $M_{1}>M_{2}$, from \eqref{3.058} and \eqref{3.059}, we have
\begin{eqnarray}
&&\sum_{i=M_{1}+1}^{\infty}\|\Phi_{1} e_{i}\|_{H^{s}}^{2}+\sum_{i=1}^{M_{1}}\|P^{M_{1}}\Phi_{1} e_{i}\|_{H^{s}}^{2}\nonumber\\
&&\leq \sum_{i=M_{1}+1}^{\infty}\|\Phi_{1} e_{i}\|_{H^{s}}^{2}+\sum_{i=1}^{M_{1}}\|P^{M_{2}}\Phi_{1} e_{i}\|_{H^{s}}^{2}< \epsilon
.\label{3.061}
\end{eqnarray}
{\bf Case 3.} When $M_{1}<M_{2}$, from \eqref{3.058} and \eqref{3.059}, we have
\begin{eqnarray}
&&\sum_{i=M_{2}+1}^{\infty}\|\Phi_{1} e_{i}\|_{H^{s}}^{2}+\sum_{i=1}^{M_{2}}\|P^{M_{2}}\Phi_{1} e_{i}\|_{H^{s}}^{2}\nonumber\\
&&\leq \sum_{i=M_{2}+1}^{\infty}\|\Phi_{1} e_{i}\|_{H^{s}}^{2}
+\sum_{i=1}^{M_{1}}\|P^{M_{2}}\Phi_{1} e_{i}\|_{H^{s}}^{2}+\sum_{i=M_{1}+1}^{M_{2}}\|\Phi_{1} e_{i}\|_{H^{s}}^{2}\nonumber\\
&&=\frac{\epsilon}{2}+\sum_{i=M_{1}+1}^{\infty}\|\Phi_{1} e_{i}\|_{H^{s}}^{2}<\epsilon
.\label{3.062}
\end{eqnarray}
Let $M=\max\{M_{1}, M_{2}\}$. Combining \eqref{3.060}, \eqref{3.061} with \eqref{3.062}, we obtain
\begin{eqnarray}
&&\sum_{i=M+1}^{\infty}\|\Phi_{1} e_{i}\|_{H^{s}}^{2}+\sum_{i=1}^{M}\|P^{M}\Phi_{1} e_{i}\|_{H^{s}}^{2}< \epsilon.\label{3.063}
\end{eqnarray}
For the above $M$ and $m\geq M$, $m\in\mathbf{N}^{+}$, we have
\begin{eqnarray}
&&\sum_{i=m+1}^{\infty}\|\Phi_{1} e_{i}\|_{H^{s}}^{2}+\sum_{i=1}^{m}\|P^{m}\Phi_{1} e_{i}\|_{H^{s}}^{2}\nonumber\\
&&\leq \sum_{i=m+1}^{\infty}\|\Phi_{1} e_{i}\|_{H^{s}}^{2}+\sum_{i=1}^{M}\|P^{m}\Phi_{1} e_{i}\|_{H^{s}}^{2}
+\sum_{i=M+1}^{m}\|\Phi_{1} e_{i}\|_{H^{s}}^{2}\nonumber\\
&&\leq\sum_{i=M+1}^{\infty}\|\Phi_{1} e_{i}\|_{H^{s}}^{2}+\sum_{i=1}^{M}\|P^{M}\Phi_{1} e_{i}\|_{H^{s}}^{2}< \epsilon,\label{3.064}
\end{eqnarray}
and we define
\begin{eqnarray*}
&&P_{m}G:=\sum_{i=1}^{\infty}\chi_{[0,T]}\int_{0}^{t}U(t-s)P_{m}\Phi_{1} e_{i}d\beta_{i}(s),\\
&&P^{m}G:=\sum_{i=1}^{\infty}\chi_{[0,T]}\int_{0}^{t}U(t-s)P^{m}\Phi_{1} e_{i}d\beta_{i}(s),\\
&&F_{m}:=\sum_{i=1}^{m}\chi_{[0,T]}\int_{0}^{t}U(t-s)P_{m}\Phi_{1} e_{i}d\beta_{i}(s).
\end{eqnarray*}
Next, for each fixed $m\geq M$, $m\in\mathbf{N}^{+}$,
we prove that $\left\|F_{m}\right\|_{X_{s,b,\varphi}}$ is $\mathcal{F}_{T}$ measurable.

\noindent We denote
\begin{eqnarray}
&&f_{i}:=\chi_{[0,T]}\int_{0}^{t}U(-s)P_{m}\Phi_{1} e_{i}d\beta_{i}(s).\label{3.065}
\end{eqnarray}
From \eqref{3.065}, we have
\begin{eqnarray}
&&\mathscr{F}_{x}f_{i}(\xi,t)
=\chi_{[0,T]}\int_{0}^{t}\chi_{\{|\xi|\leq m\}}e^{-is\xi^{3}}\mathscr{F}_{x}\Phi_{1} e_{i}(\xi)d\beta_{i}(s)\nonumber\\
&&=\chi_{[0,T]}\chi_{\{|\xi|\leq m\}}\mathscr{F}_{x}\Phi_{1} e_{i}(\xi)\int_{0}^{t}e^{-is\xi^{3}}d\beta_{i}(s).\label{3.066}
\end{eqnarray}
Since $e^{-is\xi^{3}}$ is independent of $\omega$ and continuous with respect to $(s,\xi)$, and since
\begin{eqnarray}
&&\left|\frac{d e^{-is\xi^{3}}}{ds}\right|=\left|i\xi^{3}e^{-is\xi^{3}}\right|\leq m^{3},\label{3.067}
\end{eqnarray}
it follows from \eqref{3.067} that $e^{-is\xi^{3}}$ is of bounded variation. Hence, by using Lemma 3.3, we have
\begin{eqnarray}
&&\int_{0}^{t}e^{-is\xi^{3}}d\beta_{i}(s)=e^{-it\xi^{3}}\beta_{i}(t)+i\xi^{3}\int_{0}^{t}e^{-is\xi^{3}}\beta_{i}(s)ds\nonumber\\
&&=e^{-it\xi^{3}}\beta_{i}(t)+i\xi^{3}\int_{0}^{\infty}\chi_{[0,t]}(s)e^{-is\xi^{3}}\beta_{i}(s)ds.\label{3.068}
\end{eqnarray}
By using the definition of complex-valued measurability in \cite{R1987},
$e^{-it\xi^{3}}$ is product-measurable with respect to $(t,\xi)$,
and the Brownian motion $\beta_{i}(t,\omega)$ is product-measurable with respect to $(t,\omega)$. Hence their product
$$e^{-it\xi^{3}}\beta_{i}(t)$$
is product-measurable with respect to $(t,\xi,\omega)$.

\noindent Since $\chi_{[0,t]}(s)$ is product-measurable with respect to $(t,s)$, $e^{-is\xi^{3}}$
is product-measurable with respect to $(s,\xi)$, and the Brownian motion $\beta_{i}(s,\omega)$ is product-measurable with respect to $(s,\omega)$,
it follows that
$$\chi_{[0,t]}(s)e^{-is\xi^{3}}\beta_{i}(s)$$
is product-measurable with respect to $(s,t,\xi,\omega)$. Consequently, by Fubini's theorem, we have that
$$\xi^{3}\int_{0}^{\infty}\chi_{[0,t]}(s)e^{-is\xi^{3}}\beta_{i}(s)ds$$
is product-measurable with respect to $(t,\xi,\omega)$.

\noindent From the above discussion, we have that
$$\int_{0}^{t}e^{-is\xi^{3}}d\beta_{i}(s)$$
is product-measurable with respect to $(t,\xi,\omega)$. We denote
$$\int_{0}^{t}e^{-is\xi^{3}}d\beta_{i}(s):=u_{i}(t,\xi,\omega).$$
Note that
\begin{eqnarray}
&&\left\|\chi_{[0,T]}F_{m}\right\|_{X_{s,b,\varphi}}^{2}=\left\|\sum_{i=1}^{m}f_{i}\right\|_{H_{x}^{s}H_{t}^{b}}^{2}
=\int_{\SR}\langle\xi\rangle^{2s}\left\|\sum_{i=1}^{m}\mathscr{F}_{x}f_{i}(\xi,\cdot)\right\|_{H_{t}^{b}}^{2}d\xi\nonumber\\
&&=\int_{|\xi|\leq m}\int_{\SR}\langle\xi\rangle^{2s}\left|\sum_{i=1}^{m}J_{t}^{b}\chi_{[0,T]}\mathscr{F}_{x}\Phi_{1} e_{i}(\xi)u_{i}(t,\xi,\omega)\right|^{2}dtd\xi\nonumber\\
&&=\int_{\SR}\int_{\SR}\chi_{\{|\xi|\leq m\}}\left|\sum_{i=1}^{m}\langle\xi\rangle^{s}\mathscr{F}_{x}\Phi_{1} e_{i}(\xi)J_{t}^{b}\chi_{[0,T]}u_{i}(t,\xi,\omega)\right|^{2} dtd\xi.\label{3.069}
\end{eqnarray}
By using  $\langle\xi\rangle^{s}\mathscr{F}_{x}\Phi_{1} e_{i}(\xi)\in L^{2}(\xi)$ and the definition of $L^{p}$ in
\cite{R1987}, we obtain that $\langle\xi\rangle^{s}\mathscr{F}_{x}\Phi_{1} e_{i}(\xi)$ is measurable as a function of $\xi$, and
since $u_{i}$ is measurable with respect to $(t,\xi,\omega)$, it follows that
$$\chi_{\{|\xi|\leq m\}}\left|\sum_{i=1}^{m}\langle\xi\rangle^{s}\mathscr{F}_{x}\Phi_{1} e_{i}(\xi)J_{t}^{b}\chi_{[0,T]}u_{i}(t,\xi,\omega)\right|^{2}$$
is product-measurable with respect to $(t,\xi,\omega)$. By using Fubini's theorem, we conclude that
\begin{eqnarray*}
&&\int_{\SR}\int_{\SR}\chi_{\{|\xi|\leq m\}}\left|\sum_{i=1}^{m}\langle\xi\rangle^{s}\mathscr{F}_{x}\Phi_{1} e_{i}(\xi)J_{t}^{b}\chi_{[0,T]}u_{i}(t,\xi,\omega)\right|^{2} dtd\xi
\end{eqnarray*}
is $\mathcal{F}_{T}$ measurable with respect to $\omega$.
From \eqref{3.069}, we have that $\left\|F_{m}\right\|_{X_{s,b,\varphi}}$ is $\mathcal{F}_{T}$ measurable.

\noindent By using \eqref{3.064} and Lemma 3.1, we have
\begin{eqnarray}
&&\mathbb{E}\left(\left\|G-F_{m}\right\|_{X_{s,b,\varphi}}^{2}\right)\nonumber\\
&&=\mathbb{E}\left(\left\|\sum_{i=1}^{\infty}\chi_{[0,T]}\int_{0}^{t}U(t-s)P^{m}\Phi_{1} e_{i}d\beta_{i}(s)
+\sum_{i=m+1}^{\infty}\chi_{[0,T]}\int_{0}^{t}U(t-s)P_{m}\Phi_{1} e_{i}d\beta_{i}(s)\right\|_{X_{s,b,\varphi}}^{2}\right)\nonumber\\
&&\leq \mathbb{E}\left(\left\|\sum_{i=1}^{m}\chi_{[0,T]}
\int_{0}^{t}U(t-s)P^{m}\Phi_{1} e_{i}d\beta_{i}(s)\right\|_{X_{s,b,\varphi}}^{2}\right)\nonumber\\
&&\qquad+\mathbb{E}\left(\left\|\sum_{i=m+1}^{\infty}\chi_{[0,T]}
\int_{0}^{t}U(t-s)\Phi_{1} e_{i}d\beta_{i}(s)\right\|_{X_{s,b,\varphi}}^{2}\right)\nonumber\\
&&\leq C\left(\sum_{i=1}^{m}\|P^{m}\Phi_{1} e_{i}\|_{H^{s}}^{2}
+\sum_{i=m+1}^{\infty}\|\Phi_{1} e_{i}\|_{H^{s}}^{2}\right)\nonumber\\
&&<C\epsilon.\label{3.070}
\end{eqnarray}
By using \eqref{3.070} and Chebyshev's inequality, for any $\alpha>0$(fixed), we have
\begin{eqnarray}
&&\mathbb{P}\left(\left\{\omega:\left|\left\|F_{m}\right\|_{X_{s,b}}
-\left\|G\right\|_{X_{s,b,\varphi}}\right|>\alpha\right\}\right)\nonumber\\
&&\leq \mathbb{P}\left(\left\{\omega:\left\|F_{m}-G\right\|_{X_{s,b,\varphi}}>\alpha\right\}\right)\nonumber\\
&&\leq \frac{\mathbb{E}\left(\left\|F_{m}-G\right\|_{X_{s,b,\varphi}}^{2}\right)}{\alpha^{2}}\nonumber\\
&&<C\frac{\epsilon}{\alpha^{2}}.\label{3.071}
\end{eqnarray}
From \eqref{3.071}, we have
\begin{eqnarray}
&&\left\|F_{m}\right\|_{X_{s,b,\varphi}}\overset{\mathbb{P}}{\longrightarrow}\|G\|_{X_{s,b,\varphi}}.\label{3.072}
\end{eqnarray}
From \eqref{3.072}, there exists a subsequence $\left\|F_{m_{k}}\right\|_{X_{s,b,\varphi}}$ converging to $\|G\|_{X_{s,b,\varphi}}$.
Since $\|F_{m_{k}}\|_{X_{s,b,\varphi}}$ is $\mathcal{F}_{T}$ measurable,
it follows that $\|G\|_{X_{s,b,\varphi}}$ is $\mathcal{F}_{T}$ measurable.

The proof of Lemma 3.7 is completed.

\begin{Lemma}\label{lem3.8}
Let $s\in\R$, $\Phi_{2}\in L_{2}^{0,s}$. Then, we have
\begin{eqnarray*}
&&\mathbb{P}\left(\left\{\omega: \forall t_{0}, t\in [0,T_{\omega}], T_{\omega}\in(0,1), \lim_{t\rightarrow t_{0}}
\left\|w_{2}(t,\omega,\cdot)-w_{2}(t_{0},\omega,\cdot)\right\|_{H^{s}}=0\right\}\right)\nonumber\\
&&\quad=1,
\end{eqnarray*}
where
\begin{eqnarray*}
&&w_{2}=\int_{0}^{t}V(t-s)\Phi_{2}dW(s).
\end{eqnarray*}

\end{Lemma}

By using a similar argument as in the proof of Lemma 3.6, we have that Lemma 3.8 is valid.

\begin{Lemma}\label{lem3.9}
Let $s\in\R$, $\Phi_{2}\in L_{2}^{0,s}$, $0<b<\frac{1}{2}$.
Then, for each fixed $T$,
\begin{eqnarray*}
&&\left\|\chi_{[0,T]}\int_{0}^{t}V(t-s)\Phi_{2} dW(s)\right\|_{X_{s,b,\phi}}
\end{eqnarray*}
is $\mathcal{F}_{T}$ measurable.
\end{Lemma}

By using a similar argument as in the proof of Lemma 3.7, we have that Lemma 3.9 is valid.

\bigskip

\bigskip

\section{Trilinear estimates}

\setcounter{equation}{0}

\setcounter{Theorem}{0}

\setcounter{Lemma}{0}

\setcounter{Proposition}{0}

\setcounter{section}{4}

In this section, we prove the trilinear estimates, which play an important role in proving Theorems 1.1-1.7.

\begin{Lemma}\label{lem4.1}
Let $s\geq\frac{1}{4}$,\, $0\leq a\leq
\min\left\{2s-\frac{1}{2},1-64\epsilon\right\}$,\,
$b_{1}=\frac{1}{2}-3\epsilon$,\,$0<\epsilon\leq\frac{1}{2026}$.
Then,  we have
\begin{eqnarray}
&&\|\partial_{x}(u^{3})\|_{X_{s+a, -b_{1},\varphi}}\leq C\|u\|_{X_{s,b_{1},\varphi}}^{3}.\label{4.01}
\end{eqnarray}

\end{Lemma}
\noindent{\bf Proof.}
According to the duality idea, to prove \eqref{4.01}, we only need to prove
\begin{eqnarray}
&&\left|\int_{\SR^{2}}\partial_{x}\left(\prod_{j=1}^{3}u_{j}(x,t)\right)\overline{h}(x,t)dxdt\right|\leq C\|h\|_{X_{-(s+a),b_{1},\varphi}}\prod_{j=1}^{3}\|u_{j}\|_{X_{s,b_{1}},\varphi}.\label{4.02}
\end{eqnarray}
We define
\begin{eqnarray*}
&&\sigma_{j}=\tau_{j}-\varphi(\xi_{j})=\tau_{j}-\xi_{j}^{3}(j=1,2,3),\,\sigma=\tau-\varphi(\xi)=\tau-\xi^{3},\\
&&F_{j}(\xi_{j},\tau_{j})=\langle\xi_{j}\rangle^{s}\langle\sigma_{j}\rangle^{b_{1}}\mathscr{F}_{xt}u_{j}(\xi_{j},\tau_{j})(j=1,2,3),\\
&&F(\xi,\tau)=\langle\xi\rangle^{-(s+a)}\langle\sigma\rangle^{b_{1}}\mathscr{F}_{xt}h(\xi,\tau),\\
&&\mathscr{F}_{xt}f_{j}=\frac{F_{j}}{\langle\sigma_{j}\rangle^{b_{1}}}(j=1,2,3),\,\mathscr{F}_{xt}f=\frac{F}{\langle\sigma\rangle^{b_{1}}},\\
&&K(\xi_{1},\xi_{2},\xi_{3},\xi,\tau_{1},\tau_{2},\tau_{3},\tau)=\frac{|\xi|\langle\xi\rangle^{s+a}}{\langle\sigma\rangle^{b_{1}}
\left(\prod\limits_{j=1}^{3}\langle\sigma_{j}
\rangle^{b_{1}}\right)\left(\prod\limits_{j=1}^{3}\langle\xi_{j}
\rangle^{s}\right)}.
\end{eqnarray*}
To prove \eqref{4.02}, we only need to prove
\begin{eqnarray}
&&\left|\int_{\Xi}K(\xi_{1},\xi_{2},\xi_{3},\xi,\tau_{1},\tau_{2},\tau_{3},\tau)F\prod_{j=1}^{3}F_{j}d\delta\right|\leq C\|F\|_{L_{\xi\tau}^{2}}\prod_{j=1}^{3}\|F_{j}\|_{L_{\xi\tau}^{2}},\label{4.03}
\end{eqnarray}
where
\begin{eqnarray*}
&&\int_{\Xi}d\delta=\int_{\xi=\sum\limits_{j=1}^{3}\xi_{j},\, \tau=\sum\limits_{j=1}^{3}\tau_{j}}d\xi_{1}d\tau_{1}d\xi_{2}d\tau_{2}d\xi d\tau.
\end{eqnarray*}
We denote
\begin{eqnarray*}
&&I=\int_{\Xi}K_{1}(\xi_{1},\xi_{2},\xi_{3},\xi,\tau_{1},\tau_{2},\tau_{3},\tau)F\prod_{j=1}^{3}F_{j}d\delta.
\end{eqnarray*}
Let $P_{N}F$ and $P_{N_{j}}F_{j}(1\leq j\leq3)$ be the dyadic decompositions of $F$ and $F_{j}$, respectively.
Then, we have that $\supp \mathscr{F}_{x}P_{N}F\subset\{\xi\in\R:|\xi|\sim N\}$,
and $\supp \mathscr{F}_{x}P_{N_{j}}F_{j}\subset \{\xi\in\R:|\xi|\sim N_{j}\}$, where $N$ and $N_{j}$ are dyadic numbers. $\supp \mathscr{F}_{x}P_{0}F\subset\{\xi\in\R:|\xi|\leq 2\}$,
and $\supp \mathscr{F}_{x}P_{0}F_{j}\subset \{\xi\in\R:|\xi|\leq 2\}$.
We define $\sum=\sum\limits_{N_{1}, N_{2}, N_{3}, N}$.
Without loss of generality, we assume $N_{1}\geq N_{2}\geq N_{3}\geq1$, and $N\geq 1$.

We consider the following cases: {\bf Case 1}: $N_{1}\gg N_{2}\geq N_{3}$; {\bf Case
2}: $N_{1}\sim N_{2}\gg N_{3}$ and {\bf Case 3}: $N_{1}\sim N_{2}\sim N_{3}\gg N$; {\bf Case 4}: $N_{1}\sim N_{2}\sim N_{3}\sim N$.

\noindent {\bf Case 1}: When $N_{1}\gg N_{2}\geq N_{3}$, we have $N_{1}\sim N$.
By using Corollary 2.1, Plancherel identity,  H\"{o}lder inequality, $\delta<\frac{1}{2}$
and $0\leq a\leq 1-64\epsilon$ and $s\geq0$, we have
\begin{eqnarray}
&&|I|\leq C\sum N^{1+a}N_{2}^{-s}N_{3}^{-s}\left\|P_{N_{1}}f_{1}P_{N_{2}}f_{2}\right\|_{L_{xt}^{2}}
\left\|P_{N}fP_{N_{3}}f_{3}\right\|_{L_{xt}^{2}}\nonumber\\
&&\leq C\sum N^{1+a} N_{1}^{-2(\frac{1}{2}-\delta)}N^{-2(\frac{1}{2}-\delta)}
\|P_{N}f\|_{X_{0,\frac{4-2\delta}{4}b,\varphi}}
\|P_{N_{3}}f_{3}\|_{X_{0,\frac{4-2\delta}{4}b,\varphi}}\nonumber\\
&&\quad\times \prod_{j=1}^{2}\|P_{N_{j}}f_{j}\|_{X_{0,\frac{4-2\delta}{4}b,\varphi}}
\nonumber\\
&&\leq C\sum N_{1}^{a-1+4\delta}\|P_{N}F\|_{L_{\xi\tau}^{2}}\prod_{j=1}^{3}\|P_{N_{j}}F_{j}\|_{L_{\xi\tau}^{2}}\leq C\|F\|_{L_{\xi\tau}^{2}}\prod_{j=1}^{3}\|F_{j}\|_{L_{\xi\tau}^{2}},\label{4.04}
\end{eqnarray}
where in the last inequality, we used the following condition
\begin{eqnarray}
&&\frac{4-2\delta}{4}b=\frac{4-2\delta}{4}\times\left(\frac{1}{2}+\epsilon\right)
\leq \frac{1}{2}-3\epsilon\Longleftrightarrow \delta\geq \frac{16\epsilon}{1+2\epsilon}.\label{4.05}
\end{eqnarray}
{\bf Case 2}: When $N_{1}\sim N_{2}\gg N_{3}$. By using \eqref{4.05}, \eqref{2.02}, Corollary 2.1 and Plancherel identity,
 H\"{o}lder inequality as well as $0\leq a\leq s+\frac{1}{4}-64\epsilon$, we have
\begin{eqnarray}
&&|I|\leq C\sum N^{1+a}N_{2}^{-s}N_{3}^{-s}\left\|P_{N_{1}}f_{1}P_{N_{3}}f_{3}\right\|_{L_{xt}^{2}}\left\|P_{N_{2}}f_{2}\right\|_{L_{xt}^{4}}
\left\|P_{N}f\right\|_{L_{xt}^{4}}\nonumber\\
&&\leq C\sum N^{-\frac{\epsilon}{8}}N^{1+a+\frac{\epsilon}{8}}N_{1}^{-s}N_{1}^{-2(\frac{1}{2}-\delta)}N_{2}^{-\frac{1}{8}}N^{-\frac{1}{8}}
\|P_{N_{3}}f_{3}\|_{X_{0,\frac{4-2\delta}{4}b,\varphi}}\nonumber\\
&&\quad\times\|P_{N_{1}}f_{1}\|_{X_{0,\frac{4-2\delta}{4}b,\varphi}}\|P_{N_{2}}f_{2}\|_{X_{0,\frac{3}{4}b,\varphi}}
\|P_{N}f\|_{X_{0,\frac{3}{4}b,\varphi}}\nonumber\\
&&\leq C\sum N^{-\frac{\epsilon}{8}}N_{1}^{a-s-\frac{1}{4}+\frac{\epsilon}{8}+2\delta}\|P_{N}F\|_{L_{\xi\tau}^{2}}
\prod_{j=1}^{3}\|P_{N_{j}}F_{j}\|_{L_{\xi\tau}^{2}}\nonumber\\
&&\leq C\|F\|_{L_{\xi\tau}^{2}}\prod_{j=1}^{3}\|F_{j}\|_{L_{\xi\tau}^{2}},\label{4.06}
\end{eqnarray}
where in the last inequality, we used the following condition
\begin{eqnarray}
&&\frac{3}{4}b=\frac{3}{4}\times\left(\frac{1}{2}+\epsilon\right)=\frac{3}{8}+\frac{3\epsilon}{4}<\frac{1}{2}-3\epsilon=b_{1}.\label{4.07}
\end{eqnarray}
{\bf Case 3}: When $N_{1}\sim N_{2}\sim N_{3}\gg N$.
By using \eqref{4.05}, \eqref{4.07}, \eqref{2.02}, Corollary 2.1 and Plancherel identity,  H\"{o}lder inequality
as well as  $0\leq a\leq 2s+\frac{1}{4}-64\epsilon$, we have
\begin{eqnarray}
&&|I|\leq C\sum N^{1+a}N_{2}^{-s}N_{3}^{-s}\left\|P_{N_{1}}f_{1}P_{N}f\right\|_{L_{xt}^{2}}\left\|P_{N_{2}}f_{2}\right\|_{L_{xt}^{4}}
\left\|P_{N_{3}}f_{3}\right\|_{L_{xt}^{4}}\nonumber\\
&&\leq C\sum N^{1+a}N_{1}^{-2(\frac{1}{2}-\delta)}N_{1}^{-2s}N_{2}^{-\frac{1}{8}}N_{3}^{-\frac{1}{8}}
\|P_{N}f\|_{X_{0,\frac{4-2\delta}{4}b,\varphi}}\nonumber\\
&&\quad\times\|P_{N_{1}}f_{1}\|_{X_{0,\frac{4-2\delta}{4}b,\varphi}}\prod_{j=2}^{3}\|P_{N_{j}}f_{j}\|_{X_{0,\frac{3}{4}b,\varphi}}
\nonumber\\
&&\leq C\sum N_{1}^{a-2s-\frac{1}{4}+2\delta}\|P_{N}F\|_{L_{\xi\tau}^{2}}\prod_{j=1}^{3}\|P_{N_{j}}F_{j}\|_{L_{\xi\tau}^{2}}
\leq C\|F\|_{L_{\xi\tau}^{2}}\prod_{j=1}^{3}\|F_{j}\|_{L_{\xi\tau}^{2}}.\label{4.08}
\end{eqnarray}
{\bf Case 4}: When $N_{1}\sim N_{2}\sim N_{3}\sim N$.
By using \eqref{4.07}, \eqref{2.02} and Plancherel identity,  H\"{o}lder inequality, $0\leq a\leq 2s-\frac{1}{2}$, we have
\begin{eqnarray}
&&|I|\leq C\sum N^{s+a}\left(\prod_{j=1}^{3}N_{j}^{-s}\right)
\left(\prod\limits_{j=1}^{3}\left\|P_{N_{j}}f_{j}\right\|_{L_{xt}^{4}}\right)\left\|P_{N}f\right\|_{L_{xt}^{4}}\nonumber\\
&&\leq C\sum N_{1}^{s+a-3s}N_{1}^{-\frac{1}{2}}\left(\prod\limits_{j=1}^{3}
\left\|P_{N_{j}}f_{j}\right\|_{X_{0,\frac{3}{4}b,\varphi}}\right)\|P_{N}f\|_{X_{0,\frac{3}{4}b,\varphi}}\nonumber\\
&&\leq C\sum N_{1}^{\frac{1}{2}+a-2s}\|P_{N}F\|_{L_{\xi\tau}^{2}}\prod_{j=1}^{3}\|P_{N_{j}}F_{j}\|_{L_{\xi\tau}^{2}}
\leq C\|F\|_{L_{\xi\tau}^{2}}\prod_{j=1}^{3}\|F_{j}\|_{L_{\xi\tau}^{2}}.\label{4.09}
\end{eqnarray}

The proof of Lemma 4.1 is completed.

\begin{Lemma}\label{lem4.2}
Let $s\geq\frac{1}{4}$,\, $0\leq a\leq
\min\left\{2s-\frac{1}{2},1-48\epsilon\right\}$,\, $b_{1}=\frac{1}{2}-3\epsilon$,\,$0<\epsilon\leq\frac{1}{2026}$ and $\phi(\xi)=-\nu\xi|\xi|+\xi^{3}$.
Then,  we have
\begin{eqnarray}
&&\|\partial_{x}(v^{3})\|_{X_{s+a, -b_{1},\phi}}\leq C\|v\|_{X_{s,b_{1},\phi}}^{3}.\label{4.010}
\end{eqnarray}

\end{Lemma}
\noindent{\bf Proof.}
According to the duality idea, to prove \eqref{4.010}, we only need to prove
\begin{eqnarray}
&&\left|\int_{\SR^{2}}\partial_{x}\left(\prod_{j=1}^{3}v_{j}(x,t)\right)\overline{w}(x,t)dxdt\right|\leq C\|w\|_{X_{-(s+a),b_{1},\phi}}\prod_{j=1}^{3}\|v_{j}\|_{X_{s,b_{1},\phi}}.\label{4.011}
\end{eqnarray}
We define
\begin{eqnarray*}
&&\sigma_{j}=\tau_{j}-\phi(\xi_{j})(j=1,2,3),\,\sigma=\tau-\phi(\xi),\\
&&G_{j}(\xi_{j},\tau_{j})=\langle\xi_{j}\rangle^{s}\langle\sigma_{j}\rangle^{b_{1}}\mathscr{F}_{xt}v_{j}(\xi_{j},\tau_{j})(j=1,2,3),\\
&&G(\xi,\tau)=\langle\xi\rangle^{-(s+a)}\langle\sigma\rangle^{b_{1}}\mathscr{F}_{xt}w(\xi,\tau),\\
&&\mathscr{F}_{xt}g_{j}=\frac{G_{j}}{\langle\sigma_{j}\rangle^{b_{1}}}(j=1,2,3),\,\mathscr{F}_{xt}g=\frac{G}{\langle\sigma\rangle^{b_{1}}},\\
&&K_{1}(\xi_{1},\xi_{2},\xi_{3},\xi,\tau_{1},\tau_{2},\tau_{3},\tau)=\frac{|\xi|\langle\xi\rangle^{s+a}}{\langle\sigma\rangle^{b_{1}}
\left(\prod\limits_{j=1}^{3}\langle\sigma_{j}
\rangle^{b_{1}}\right)\left(\prod\limits_{j=1}^{3}\langle\xi_{j}
\rangle^{s}\right)}.
\end{eqnarray*}
To prove \eqref{4.011}, we only need to prove
\begin{eqnarray}
&&\left|\int_{\Xi}K_{1}(\xi_{1},\xi_{2},\xi_{3},\xi,\tau_{1},\tau_{2},\tau_{3},\tau)G\prod_{j=1}^{3}G_{j}d\delta\right|\leq C\|G\|_{L_{\xi\tau}^{2}}\prod_{j=1}^{3}\|G_{j}\|_{L_{\xi\tau}^{2}},\label{4.012}
\end{eqnarray}
where
\begin{eqnarray*}
&&\int_{\Xi}d\delta=\int_{\xi=\sum\limits_{j=1}^{3}\xi_{j},\, \tau=\sum\limits_{j=1}^{3}\tau_{j}}d\xi_{1}d\tau_{1}d\xi_{2}d\tau_{2}d\xi d\tau.
\end{eqnarray*}
We denote
\begin{eqnarray*}
&&I_{1}=\int_{\Xi}K_{3}(\xi_{1},\xi_{2},\xi_{3},\xi,\tau_{1},\tau_{2},\tau_{3},\tau)G\prod_{j=1}^{3}G_{j}d\delta.
\end{eqnarray*}
Let $P_{N}G$ and $P_{N_{j}}G_{j}(1\leq j\leq3)$ be the dyadic decompositions of $G$ and $G_{j}$, respectively.
Then, we have that $\supp \mathscr{F}_{x}P_{N}G\subset\{\xi\in\R:|\xi|\sim N\}$,
and $\supp \mathscr{F}_{x}P_{N_{j}}G_{j}\subset \{\xi\in\R:|\xi|\sim N_{j}\}$, where $N$ and $N_{j}$ are dyadic numbers. $\supp \mathscr{F}_{x}P_{0}G\subset\{\xi\in\R:|\xi|\leq 2\}$,
and $\supp \mathscr{F}_{x}P_{0}G_{j}\subset \{\xi\in\R:|\xi|\leq 2\}$.
We define $\sum=\sum\limits_{N_{1}, N_{2}, N_{3}, N}$.
Without loss of generality, we assume $N_{1}\geq N_{2}\geq N_{3}\geq 2A$, and $N\geq 2A$.

We consider the following cases: {\bf Case 1}: $N_{1}\gg N_{2}\geq N_{3}$; {\bf Case
2}: $N_{1}\sim N_{2}\gg N_{3}$ and {\bf Case 3}: $N_{1}\sim N_{2}\sim N_{3}\gg N$; {\bf Case 4}: $N_{1}\sim N_{2}\sim N_{3}\sim N$.

\noindent {\bf Case 1}: When $N_{1}\gg N_{2}\geq N_{3}$, we have $N_{1}\sim N$.
By using  Corollary 2.2, Plancherel identity,  H\"{o}lder inequality, $\delta<\frac{1}{2}$
and $0\leq a\leq 1-48\epsilon$ and $s\geq0$, we have
\begin{eqnarray}
&&|I_{1}|\leq C\sum N^{1+a}N_{2}^{-s}N_{3}^{-s}\left\|P_{N_{1}}g_{1}P_{N_{2}}g_{2}\right\|_{L_{xt}^{2}}\left\|P_{N}gP_{N_{3}}g_{3}\right\|_{L_{xt}^{2}}\nonumber\\
&&\leq C\sum N^{1+a} N_{1}^{-2(\frac{1}{2}-\delta)}N^{-2(\frac{1}{2}-\delta)}\|P_{N}g\|_{X_{0,\frac{3-2\delta}{3}b,\phi}}
\|P_{N_{3}}g_{3}\|_{X_{0,\frac{3-2\delta}{3}b,\phi}}\nonumber\\
&&\quad\times\prod_{j=1}^{2}\|g_{j}\|_{X_{0,\frac{3-2\delta}{3}b,\phi}}\nonumber\\
&&\leq C\sum N_{1}^{a-1+4\delta}\|P_{N}G\|_{L_{\xi\tau}^{2}}\prod_{j=1}^{3}\|P_{N_{j}}G_{j}\|_{L_{\xi\tau}^{2}}
\leq C\|G\|_{L_{\xi\tau}^{2}}\prod_{j=1}^{3}\|G_{j}\|_{L_{\xi\tau}^{2}},\label{4.013}
\end{eqnarray}
where in the last inequality, we used the following condition
\begin{eqnarray}
&&\frac{3-2\delta}{3}b=\frac{3-2\delta}{3}\times\left(\frac{1}{2}+\epsilon\right)
\leq \frac{1}{2}-3\epsilon\Longleftrightarrow \delta\geq \frac{12\epsilon}{1+2\epsilon}.\label{4.014}
\end{eqnarray}
{\bf Case 2}: When $N_{1}\sim N_{2}\gg N_{3}$.
By using \eqref{4.014}, \eqref{2.010}, Corollary 2.2 and Plancherel identity,  H\"{o}lder inequality
as well as $0\leq a\leq s+\frac{1}{4}-48\epsilon$, we have
\begin{eqnarray}
&&|I_{1}|\leq C\sum N^{1+a}N_{2}^{-s}N_{3}^{-s}\left\|P_{N_{1}}g_{1}P_{N_{3}}g_{3}\right\|_{L_{xt}^{2}}\left\|P_{N_{2}}g_{2}\right\|_{L_{xt}^{4}}
\left\|P_{N}g\right\|_{L_{xt}^{4}}\nonumber\\
&&\leq C\sum N^{-\frac{\epsilon}{8}}N^{1+a+\frac{\epsilon}{8}}N_{1}^{-s}N_{1}^{-2(\frac{1}{2}-\delta)}N_{2}^{-\frac{1}{8}}N^{-\frac{1}{8}}
\|P_{N_{3}}g_{3}\|_{X_{0,\frac{3-2\delta}{3}b,\phi}}\nonumber\\
&&\quad\times\|P_{N_{1}}g_{1}\|_{X_{0,\frac{3-2\delta}{3}b,\phi}}\|P_{N_{2}}g_{2}\|_{X_{0,\frac{3}{4}b,\phi}}
\|P_{N}g\|_{X_{0,\frac{3}{4}b,\phi}}\nonumber\\
&&\leq C\sum N^{-\frac{\epsilon}{8}}N_{1}^{a-s-\frac{1}{4}+\frac{\epsilon}{8}+2\delta}\|P_{N}G\|_{L_{\xi\tau}^{2}}
\prod_{j=1}^{3}\|P_{N_{j}}G_{j}\|_{L_{\xi\tau}^{2}}\nonumber\\
&&\leq C\|G\|_{L_{\xi\tau}^{2}}\prod_{j=1}^{3}\|G_{j}\|_{L_{\xi\tau}^{2}},\label{4.015}
\end{eqnarray}
where in the last inequality, we used the following condition
\begin{eqnarray}
&&\frac{3}{4}b=\frac{3}{4}\times\left(\frac{1}{2}+\epsilon\right)=\frac{3}{8}+\frac{3\epsilon}{4}<\frac{1}{2}-3\epsilon=b_{1},\label{4.016}
\end{eqnarray}
{\bf Case 3}: When $N_{1}\sim N_{2}\sim N_{3}\gg N$.
By using \eqref{4.014}, \eqref{4.016}, \eqref{2.010}, Corollary 2.2 and Plancherel identity,  H\"{o}lder inequality as well as $0\leq a\leq 2s+\frac{1}{4}-48\epsilon$,  we have
\begin{eqnarray}
&&|I_{1}|\leq C\sum N^{1+a}N_{2}^{-s}N_{3}^{-s}\left\|P_{N_{1}}g_{1}P_{N}g\right\|_{L_{xt}^{2}}\left\|P_{N_{2}}f_{2}\right\|_{L_{xt}^{4}}
\left\|P_{N_{3}}f_{3}\right\|_{L_{xt}^{4}}\nonumber\\
&&\leq C\sum N^{1+a}N_{1}^{-2s}N_{1}^{-2(\frac{1}{2}-\delta)}N_{2}^{-\frac{1}{8}}N_{3}^{-\frac{1}{8}}
\|P_{N}g\|_{X_{0,\frac{3-2\delta}{3}b,\phi}}\nonumber\\
&&\quad\times\|P_{N_{1}}g_{1}\|_{X_{0,\frac{3-2\delta}{3}b,\phi}}\prod_{j=2}^{3}\|P_{N_{j}}g_{j}\|_{X_{0,\frac{3}{4}b,\phi}}\nonumber\\
&&\leq C\sum N_{1}^{a-2s-\frac{1}{4}+2\delta}\|P_{N}G\|_{L_{\xi\tau}^{2}}\prod_{j=1}^{3}\|P_{N_{j}}G_{j}\|_{L_{\xi\tau}^{2}}
\leq C\|G\|_{L_{\xi\tau}^{2}}\prod_{j=1}^{3}\|G_{j}\|_{L_{\xi\tau}^{2}}.\label{4.017}
\end{eqnarray}
{\bf Case 4}: When $N_{1}\sim N_{2}\sim N_{3}\sim N$.
By using \eqref{4.016}, \eqref{2.010} and Plancherel identity,  H\"{o}lder inequality, $0\leq a\leq 2s-\frac{1}{2}$, we have
\begin{eqnarray}
&&|I_{1}|\leq C\sum N^{s+a}\left(\prod_{j=1}^{3}N_{j}^{-s}\right)
\left(\prod\limits_{j=1}^{3}\left\|P_{N_{j}}g_{j}\right\|_{L_{xt}^{4}}\right)\left\|P_{N}g\right\|_{L_{xt}^{4}}\nonumber\\
&&\leq C\sum N_{1}^{s+a-3s}N_{1}^{-\frac{1}{2}}\left(\prod\limits_{j=1}^{3}
\left\|P_{N_{j}}g_{j}\right\|_{X_{0,\frac{3}{4}b,\phi}}\right)\|P_{N}g\|_{X_{0,\frac{3}{4}b,\phi}}\nonumber\\
&&\leq C\sum N_{1}^{\frac{1}{2}+a-2s}\|P_{N}G\|_{L_{\xi\tau}^{2}}\prod_{j=1}^{3}\|P_{N_{j}}G_{j}\|_{L_{\xi\tau}^{2}}
\leq C\|G\|_{L_{\xi\tau}^{2}}\prod_{j=1}^{3}\|G_{j}\|_{L_{\xi\tau}^{2}}.\label{4.018}
\end{eqnarray}

The proof of Lemma 4.2 is completed.

\begin{Lemma}\label{lem4.3}
Let $s\geq\frac{1}{4}$,\, $0\leq a\leq
\min\left\{2s-\frac{1}{2},1-64\epsilon\right\}$,\,$b_{1}=\frac{1}{2}-3\epsilon$,\,$0<\epsilon\leq\frac{1}{2026}$.
Then,  we have
\begin{eqnarray}
&&\|\partial_{x}(u^{3})\|_{X_{s+a, -b_{1},\varphi}^{T}}\leq C\|u\|_{X_{s,b_{1},\varphi}^{T}}^{3}.\label{4.019}
\end{eqnarray}

\end{Lemma}

\noindent{\bf Proof.}
According to the definition of the $X_{s,b_1,\varphi}^{T}$ norm, for any $n\in\mathbf{N}$, there exists a sequence $\{f_{n}\}$ defined on the whole $\R$ such that:
\begin{eqnarray}
&&f_{n}=u,\,t\in[0,T],\label{4.020}\\
&&\left\|f_{n}\right\|_{X_{s,b_{1},\varphi}}\leq \|u\|_{X_{s,b_{1},\varphi}^{T}}+\frac{1}{n}.\label{4.021}
\end{eqnarray}
By using Lemma 4.1, we have
\begin{eqnarray}
&&\|\partial_{x}[(f_{n})^{3}]\|_{X_{s+a, -b_{1},\varphi}}\leq C\|f_{n}\|_{X_{s,b_{1},\varphi}}^{3}.\label{4.022}
\end{eqnarray}
From \eqref{4.020}, we have
\begin{eqnarray}
&&\partial_{x}[(f_{n})^{3}]=\partial_{x}(u^{3}),\,t\in[0,T],\label{4.023}
\end{eqnarray}
From \eqref{4.023}, we know that $\partial_{x}[(f_{n})^{3}]$ is an extension of $\partial_{x}(u^{3})$ to the whole space $\R$.
According to the definition of the $X_{s,b_1,\varphi}^{T}$ norm, \eqref{4.021} and \eqref{4.022}, we have
\begin{eqnarray}
&&\|\partial_{x}(u^{3})\|_{X_{s+a, -b_{1},\varphi}^{T}}\leq \|\partial_{x}[(f_{n})^{3}]\|_{X_{s+a,-b_{1},\varphi}}
\leq C\|f_{n}\|_{X_{s,b_{1},\varphi}}^{3}\nonumber\\
&&\leq C\left(\|u\|_{X_{s,b_{1},\varphi}^{T}}+\frac{1}{n}\right)^{3}.\label{4.024}
\end{eqnarray}
By using \eqref{4.024}, we have
\begin{eqnarray}
&&\|\partial_{x}(u^{3})\|_{X_{s+a, -b_{1},\varphi}^{T}}\leq C\left(\|u\|_{X_{s,b_{1},\varphi}^{T}}+\frac{1}{n}\right)^{3}.\label{4.025}
\end{eqnarray}
Taking the $\lim\limits_{n\rightarrow\infty}$ on both sides of \eqref{4.025}, we obtain
\begin{eqnarray}
&&\|\partial_{x}(u^{3})\|_{X_{s+a, -b_{1},\varphi}^{T}}\leq C\|u\|_{X_{s,b_{1},\varphi}^{T}}^{3}.\label{4.026}
\end{eqnarray}

The proof of Lemma 4.3 is completed.

\begin{Lemma}\label{lem4.4}
Let $s\geq\frac{1}{4}$,\, $0\leq a\leq
\min\left\{2s-\frac{1}{2},1-48\epsilon\right\}$,\,
$b_{1}=\frac{1}{2}-3\epsilon$,\,$0<\epsilon\leq\frac{1}{2026}$ and $\phi(\xi)=-\nu\xi|\xi|+\xi^{3}$.
Then,  we have
\begin{eqnarray}
&&\|\partial_{x}(v^{3})\|_{X_{s+a, -b_{1},\phi}^{T}}\leq C\|v\|_{X_{s,b_{1},\phi}^{T}}^{3}.\label{4.027}
\end{eqnarray}

\end{Lemma}

Using a proof method similar to that of Lemma 4.3, we can prove Lemma 4.4.

\bigskip

\bigskip

\section{Proof of Theorem  1.1}

\setcounter{equation}{0}

 \setcounter{Theorem}{0}

\setcounter{Lemma}{0}

\setcounter{Proposition}{0}

\setcounter{section}{5}
In this section, we prove Theorem 1.1.

\noindent Using Duhamel's principle, we can see that \eqref{1.01}-\eqref{1.02} is equivalent to the following integral equation
\begin{eqnarray}
&&u=U(t)f-\int_{0}^{t}U(t-s)\partial_{x}(u^{3})ds+\int_{0}^{t}U(t-s)\Phi_{1} dW(s).\label{5.01}
\end{eqnarray}
We define
\begin{eqnarray}
&&z:=U(t)f,\, w_{1}:=\int_{0}^{t}U(t-s)\Phi_{1} dW(s),\, u_{1}:=u-z-w_{1}.\label{5.02}
\end{eqnarray}
From \eqref{5.01}-\eqref{5.02}, we have
\begin{eqnarray*}
&&u_{1}(t)=-\int_{0}^{t}U(t-s)\partial_{x}\left[(u_{1}+z+w_{1})^{3}\right]ds.
\end{eqnarray*}
For $s\geq\frac{1}{4}$, $f\in H^{s}(\R)$ and $\Phi_{1}\in L_{2}^{0,s}$, by using \cite[Theorem1.1]{YHG2021},
we have that for almost surely $\omega\in \Omega$, there exist a $T_{\omega}>0$
such that $u\in C([0,T_{\omega}]; H^{s}(\R))\cap X_{s,b}^{T_{\omega}}$.

\noindent By using Lemma 2.6 and Lemma 3.1, we have
\begin{eqnarray}
&&\left\|\chi_{[0,T]}  w_{1}\right\|_{X_{s,b_{1},\varphi}}\leq C\left\|\psi(t) w_{1}\right\|_{X_{s,b_{1},\varphi}}<\infty, a.s..\label{5.03}
\end{eqnarray}
For $b_{1}=\frac{1}{2}-3\epsilon$, by using \eqref{2.018}, \cite[Lemma 3.3]{KPV1993Duke}, \eqref{5.03} and Lemma 4.3, we have
\begin{eqnarray}
&&\left\|u_{1}\right\|_{X_{s+a,b,\varphi}^{T_{\omega}}}=
\left\|\int_{0}^{t}U(t-s)\partial_{x}\left[\left(u_{1}+U(t)f+w_{1}\right)^{3}\right]ds\right\|_{X_{s+a,\frac{1}{2}+\epsilon,\varphi}^{T_{\omega}}}\nonumber\\
&&\leq CT_{\omega}^{-\epsilon}\left\|\partial_{x}\left[\left(U(t)f+u_{1}+w_{1}\right)^{3}\right]
\right\|_{X_{s+a,-\frac{1}{2}+\epsilon,\varphi}^{T_{\omega}}}\nonumber\\
&&\leq CT_{\omega}^{-\epsilon+\frac{1}{2}-\epsilon-b_{1}-\epsilon}\left\|\partial_{x}\left[\left(U(t)f+u_{1}+w_{1}\right)^{3}\right]
\right\|_{X_{s+a,-b_{1},\varphi}^{T_{\omega}}}\nonumber\\
&&=C\left\|\partial_{x}\left[\left(U(t)f+u_{1}+w_{1}\right)^{3}\right]\right\|_{X_{s+a,-b_{1},\varphi}^{T_{\omega}}}\nonumber\\
&&\leq C\left(\|U(t)f\|_{X_{s,b_{1},\varphi}^{T_{\omega}}}^{3}+\|u_{1}\|_{X_{s,b_{1},\varphi}^{T_{\omega}}}^{3}+\left\|\psi(t) w_{1}\right\|_{X_{s,b_{1},\varphi}}^{3}\right)\nonumber\\
&&\leq C\left(1+\|f\|_{H^{s}}^{3}+\left\|\psi(t) w_{1}\right\|_{X_{s,b_{1},\varphi}}^{3}\right)<\infty, \,a.s..\label{5.04}
\end{eqnarray}
From \eqref{5.04}, we have that \eqref{1.05} is valid.

The proof of Theorem 1.1 is completed.

\bigskip

\section{Proof of Theorem  1.2}

\setcounter{equation}{0}

 \setcounter{Theorem}{0}

\setcounter{Lemma}{0}

\setcounter{Proposition}{0}

\setcounter{section}{6}
In this section, we prove Theorem 1.2.

\noindent By using \cite[Theorem 1.1]{YYY2026}, \eqref{5.04} and  $a=2s-\frac{1}{2}$, $s>\frac{1}{3}$, we have
\begin{eqnarray}
&&\|u_{1}\|_{C([0,T_{\omega}];L_{x}^{\infty})}\leq C\left\|u_{1}\right\|_{X_{s+a,\frac{1}{2}+\epsilon,\varphi}^{T_{\omega}}}\nonumber\\
&&\leq C\left(1+\|f\|_{H^{s}}^{3}+\left\|\psi(t) w_{1}\right\|_{X_{s,b_{1},\varphi}}^{3}\right)<\infty, \,a.s..\label{6.01}
\end{eqnarray}
From \eqref{7.01}, for almost every $\omega\in \Omega$, we have
\begin{eqnarray}
&&\lim_{t\rightarrow0}\|u_{1}\|_{L_{x}^{\infty}}=0.\label{6.02}
\end{eqnarray}
It follows from \eqref{6.02} that
\begin{eqnarray}
&&\mathbb{P}\left(\left\{\omega: \lim_{t\rightarrow0}\|u_{1}\|_{L_{x}^{\infty}}=0\right\}\right)=1.\label{6.03}
\end{eqnarray}

The proof of Theorem 1.2 is completed.

\bigskip

\section{Proof of Theorem  1.3}

\setcounter{equation}{0}

 \setcounter{Theorem}{0}

\setcounter{Lemma}{0}

\setcounter{Proposition}{0}

\setcounter{section}{7}
In this section, we prove Theorem 1.3.

\noindent By using \cite[Lemma 3.6]{YYY2026}, \eqref{5.04} and  $a=2s-\frac{1}{2}$, $s>\frac{1}{3}$, we have
\begin{eqnarray}
&&\left\|u_{1}\right\|_{C([0,T_{\omega}];H^{\frac{1}{2}+\epsilon})}\leq C\left\|u_{1}\right\|_{X_{s+a,\frac{1}{2}+\epsilon,\varphi}^{T_{\omega}}}\nonumber\\
&&\leq C\left(1+\|f\|_{H^{s}}^{3}+\left\|\psi(t) w_{1}\right\|_{X_{s,b_{1},\varphi}}^{3}\right)<\infty, \,a.s..\label{7.01}
\end{eqnarray}
From \eqref{7.01}, for almost every $\omega\in \Omega$ and all $t\in[0,T_{\omega}]$, we have
\begin{eqnarray}
&&\lim_{|x|\rightarrow\infty}u_{1}=-\lim_{|x|\rightarrow\infty}\int_{0}^{t}U(t-s)\partial_{x}\left[(u_{1}+U(t)f+w_{1})^{3}\right]ds=0.\label{7.02}
\end{eqnarray}
It follows from \eqref{7.02} that
\begin{eqnarray}
&&\mathbb{P}\left(\left\{\omega: \forall t\in[0,T_{\omega}], \lim_{|x|\rightarrow\infty}u_{1}=0\right\}\right)=1.\label{7.03}
\end{eqnarray}

The proof of Theorem 1.3 is completed.

\bigskip

\section{Proof of Theorem  1.4}

\setcounter{equation}{0}

 \setcounter{Theorem}{0}

\setcounter{Lemma}{0}

\setcounter{Proposition}{0}

\setcounter{section}{8}
In this section, we prove Theorem 1.4.

\noindent Using Duhamel's principle, we can see that \eqref{1.03}-\eqref{1.04} is equivalent to the following integral equation
\begin{eqnarray}
&&v=V(t)g-\int_{0}^{t}V(t-s)\partial_{x}(v^{3})ds+\int_{0}^{t}V(t-s)\Phi_{2} dW(s).\label{8.01}
\end{eqnarray}
We define
\begin{eqnarray}
&&z_{1}:=V(t)g,\, w_{2}:=\int_{0}^{t}V(t-s)\Phi_{2} dW(s),\, v_{1}:=v-z_{1}-w_{2}.\label{8.02}
\end{eqnarray}
From \eqref{8.01}-\eqref{8.02}, we have
\begin{eqnarray*}
&&v_{1}(t)=-\int_{0}^{t}V(t-s)\partial_{x}\left[(v_{1}+z_{1}+w_{2})^{3}\right]ds.
\end{eqnarray*}
We also define
\begin{eqnarray*}
&&B_{T}:=\left\{v_{1}:\|v_{1}\|_{X_{s,b_{1},\phi}^{T}}\leq 1\right\},\\
&&\Gamma_{2}(v_{1}):=-\int_{0}^{t}V(t-s)\partial_{x}\left[(v_{1}+z_{1}+w_{2})^{3}\right]ds,
\end{eqnarray*}
where $0<T<1$.
Now, we prove that there exists a stopping time $T_{\omega}>0$ such that $\Gamma_{2}$ is a contraction mapping on $B_{T_{\omega}}$.

On the one hand, for $v_{1}\in B_{T}$, by using \eqref{2.020}, \eqref{2.021} and Lemma 2.6 and Lemma 4.4 with $a=0$, we have
\begin{eqnarray}
&&\left\|\Gamma_{2}(v_{1})\right\|_{X_{s,b_{1},\phi}^{T}}
=\left\|\int_{0}^{t}V(t-s)\partial_{x}\left[(v_{1}+z_{1}+w_{2})^{3}\right]ds\right\|_{X_{s,b_{1},\phi}^{T}}
\nonumber\\
&&\leq CT^{6\epsilon}\left(\|z_{1}\|_{X_{s,b_{1},\phi}^{T}}^{3}
+\|v_{1}\|_{X_{s,b_{1},\phi}^{T}}^{3}+\left\|w_{2}\right\|_{X_{s,b_{1},\phi}^{T}}^{3}\right)\nonumber\\
&&\leq CT^{6\epsilon}\left(\|z_{1}\|_{X_{s,b_{1},\phi}^{T}}^{3}+\|v_{1}\|_{X_{s,b_{1},\phi}^{T}}^{3}+\left\|\chi_{[0,T]} w_{2}\right\|_{X_{s,b_{1},\phi}}^{3}\right)\nonumber\\
&&\leq CT^{6\epsilon}\left(\|g\|_{H^{s}}^{3}+\|v_{1}\|_{X_{s,b_{1},\phi}^{T}}^{3}
+\left\|\chi_{[0,T]} w_{2}\right\|_{X_{s,b_{1},\phi}}^{3}\right)\nonumber\\
&&\leq CT^{6\epsilon}\left(1+\|g\|_{H^{s}}^{3}+\left\|\chi_{[0,T]} w_{2}\right\|_{X_{s,b_{1},\phi}}^{3}\right).\label{8.03}
\end{eqnarray}
On the other hand, for $v_{1}, v_{2} \in B_{T}$, by using \eqref{2.020}, \eqref{2.021} and Lemma 2.6 and Lemma 4.4 with $a=0$, we have
\begin{eqnarray}
&&\left\|\Gamma_{2}(v_{1})-\Gamma_{2}(v_{2})\right\|_{X_{s,b_{1},\phi}^{T}}=
\left\|\int_{0}^{t}V(t-s)\partial_{x}\left[(v_{1}+z_{1}+w_{2})^{3}-(v_{2}+z_{1}+w_{2})^{3}\right]ds\right\|_{X_{s,b_{1},\phi}^{T}}\nonumber\\
&&\leq CT^{6\epsilon}\left(\|z_{1}\|_{X_{s,b_{1},\phi}^{T}}^{2}+\|v_{1}\|_{X_{s,b_{1},\phi}^{T}}^{2}
+\|v_{2}\|_{X_{s,b_{1},\phi}^{T}}^{2}
+\left\| w_{2}\right\|_{X_{s,b_{1},\phi}^{T}}^{2}\right)\left(\|v_{1}-v_{2}\|_{X_{s,b_{1},\phi}^{T}}\right)\nonumber\\
&&\leq CT^{6\epsilon}\left(\|z_{1}\|_{X_{s,b_{1},\phi}^{T}}^{2}
+\|v_{1}\|_{X_{s,b_{1},\phi}^{T}}^{2}+\|v_{2}\|_{X_{s,b_{1},\phi}^{T}}^{2}+\left\|\chi_{[0,T]} w_{2}\right\|_{X_{s,b_{1},\phi}}^{2}\right)\left(\|v_{1}-v_{2}\|_{X_{s,b_{1},\phi}^{T}}\right)\nonumber\\
&&\leq CT^{6\epsilon}\left(2+\|g\|_{H^{s}}^{2}+\left\|\chi_{[0,T]} w_{2}\right\|_{X_{s,b_{1},\phi}}^{2}\right)\left(\|v_{1}-v_{2}\|_{X_{s,b_{1},\phi}^{T}}\right).\label{8.04}
\end{eqnarray}
We define
\begin{eqnarray*}
&&E_{\omega}:=\left(2+\|g\|_{H^{s}}+\left\|\chi_{[0,T]} w_{2}\right\|_{X_{s,b_{1},\phi}}\right),\, T_{\omega}:=\inf\left\{0<T<1: CT^{6\epsilon}E_{\omega}^{3}\geq\frac{1}{4}\right\}.
\end{eqnarray*}
\subsection{Proof of stopping time $T_{\omega}$ and that $\mathbb{P}\left(\{\omega: 0<T_{\omega}<1\}\right)=1$}

By using Lemma 2.6 and Lemma 3.2, we have
\begin{eqnarray}
&&\left\|\chi_{[0,T]} w_{2}\right\|_{X_{s,b_{1},\phi}}\leq C\left\|\psi(t) w_{2}\right\|_{X_{s,b_{1},\phi}}<\infty, a.s..\label{8.05}
\end{eqnarray}
It follows from Lemma 2.9, with $b_{1}=\frac{1}{2}-3\epsilon$, that the norm $\left\|\chi_{[0,T]} w_{2}\right\|_{X_{s,b_{1},\phi}}$
is almost surely continuous with respect to $T\in(0,1)$. Then, we have
\begin{eqnarray}
&&\mathbb{P}\left(\left\{\omega: \lim_{T\rightarrow0^{+}}\left\|\chi_{[0,T]} w_{2}\right\|_{X_{s,b_{1},\phi}}=0\right\}\right)=1.\label{8.06}
\end{eqnarray}
From \eqref{8.06}, we have
\begin{eqnarray}
&&\mathbb{P}\left(\left\{\omega: \lim_{T\rightarrow0^{+}}CT^{6\epsilon}E_{\omega}^{3}=0\right\}\right)=1.\label{8.07}
\end{eqnarray}
From \eqref{8.07}, we know that for any $\epsilon_{1}>0$, there exists a $\delta(\epsilon_{1})>0$ such that for all $0 <T<\delta(\epsilon_{1})$,
\begin{eqnarray}
&&CT^{6\epsilon}E_{\omega}^{3}\leq\epsilon_{1}.\label{8.08}
\end{eqnarray}
From \eqref{8.08} and the definition of $T_{\omega}$, we have
\begin{eqnarray*}
&&\mathbb{P}\left(\{\omega: 0<T_{\omega}<1\}\right)=1.
\end{eqnarray*}

\noindent We define
\begin{eqnarray}
&&K(T,\omega):=CT^{6\epsilon}E_{\omega}^{3}-\frac{1}{4}.\label{8.09}
\end{eqnarray}
By using \eqref{8.09}, we have
\begin{eqnarray}
&&T_{\omega}=\inf\left\{0<T<1: K(T,\omega)\geq0\right\}.\label{8.010}
\end{eqnarray}
Since $K(T,\omega)$ is almost surely continuous with respect to $T$, we have
\begin{eqnarray}
&&\{T_{\omega}\leq t\}=\bigcup_{q\in\mathbf{Q}\cap(0,t]}\{K(q,\omega)\geq0\}.\label{8.011}
\end{eqnarray}
Moreover, by using Lemma 3.9, we have that
 $\|\chi_{[0,T]} w_{2}\|_{X_{s,b_{1},\phi}}$ is $\mathcal{F}_{T}$-measurable, and it follows that for any rational number $q\in(0,t]$, we have
\begin{eqnarray}
&&\{K(q,\omega)\geq0\}\in\mathcal{F}_{q}\subset\mathcal{F}_{t}.\label{8.012}
\end{eqnarray}
From \eqref{8.011} and \eqref{8.012}, we have that $T_{\omega}$ is a stopping time.

\subsection{Contraction property of $\Gamma_{2}$}
From \eqref{8.03}-\eqref{8.04} and the definition of $T_{\omega}$,  it follows that  $\Gamma_{2}$ is a contraction mapping on $B_{T_{\omega}}$.

\subsection{Proof of $\mathbb{P}\left(\{\omega: v\in C([0,T_{\omega}]; H^{s}(\R))\}\right)=1$}
In this section, we prove that $v\in C([0,T_{\omega}]; H^{s}(\R))$ almost surely.  For $z$, by using \cite[Lemma 3.6]{YYY2026}, \eqref{2.020} and $b_{1}=\frac{1}{2}-3\epsilon$, we obtain
\begin{eqnarray}
&&\|z_{1}\|_{C([0,T_{\omega}]; H^{s})}\leq \|z_{1}\|_{X_{s,1-b_{1},\phi}}\leq C.\label{8.013}
\end{eqnarray}
From \eqref{8.013},  it follows that $z_{1}\in C([0,T_{\omega}]; H^{s}(\R))$ almost surely.

\noindent For $v_{1}$, by using \eqref{2.020}, \cite[Lemma 3.6]{YYY2026}, \cite[Lemma 3.3]{KPV1993Duke}, \eqref{8.05} and  Lemma 4.4 with $a=0$, $b_{1}=\frac{1}{2}-3\epsilon$, we have
\begin{eqnarray}
&&\left\|v_{1}\right\|_{C([0,T_{\omega}];H^{s})}\leq C\left\|\int_{0}^{t}V(t-s)\partial_{x}\left[(v_{1}+z_{1}+w_{2})^{3}\right]ds\right\|_{X_{s,\frac{1}{2}+\epsilon,\phi}^{T_{\omega}}}\nonumber\\
&&\leq CT_{\omega}^{-\epsilon}\left\|\partial_{x}\left[(z_{1}+v_{1}+w_{2})^{3}\right]\right\|_{X_{s,-\frac{1}{2}+\epsilon,\phi}^{T_{\omega}}}\nonumber\\
&&\leq CT_{\omega}^{-\epsilon+\frac{1}{2}-\epsilon-b_{1}-\epsilon}\left\|\partial_{x}\left[(z_{1}+v_{1}+w_{2})^{3}\right]\right\|_{X_{s,-b_{1},\phi}
^{T_{\omega}}}\nonumber\\
&&=C\left\|\partial_{x}\left[(z_{1}+v_{1}+w_{2})^{3}\right]\right\|_{X_{s,-b_{1},\phi}^{T_{\omega}}}\nonumber\\
&&\leq C\left(\|z_{1}\|_{X_{s,b_{1},\phi}^{T_{\omega}}}^{3}+\|v_{1}\|_{X_{s,b_{1},\phi}^{T_{\omega}}}^{3}+\left\|\psi(t) w_{2}\right\|_{X_{s,b_{1},\phi}}^{3}\right)\nonumber\\
&&\leq C\left(1+\|g\|_{H^{s}}^{3}+\left\|\psi(t) w_{2}\right\|_{X_{s,b_{1},\phi}}^{3}\right)<\infty,\,a.s..\label{8.014}
\end{eqnarray}
By using \eqref{8.014}, we have that
\begin{eqnarray*}
&&\mathbb{P}\left(\left\{\omega: v_{1}\in C([0,T_{\omega}]; H^{s}(\R))\right\}\right)=1.
\end{eqnarray*}
For $w_{2}$, by using $\Phi_{2}\in L_{2}^{0,s}$ and Lemma 3.8, we have that
\begin{eqnarray*}
&&\mathbb{P}\left(\left\{\omega: w_{2}\in C([0,T_{\omega}]; H^{s}(\R))\right\}\right)=1.
\end{eqnarray*}
From the above discussion, we obtain that
\begin{eqnarray*}
&&\mathbb{P}\left(\left\{\omega: v\in C([0,T_{\omega}]; H^{s}(\R))\right\}\right)=1.
\end{eqnarray*}
For the proof of the remaining part of Theorem 1.5, we refer to  \cite{BDT1999,R2012,R2014}.

The proof of Theorem 1.4 is completed.

\bigskip

\section{Proof of Theorem  1.5}

\setcounter{equation}{0}

 \setcounter{Theorem}{0}

\setcounter{Lemma}{0}

\setcounter{Proposition}{0}

\setcounter{section}{9}
In this section, we prove Theorem 1.5.

For $b_{1}=\frac{1}{2}-3\epsilon$, by using \eqref{2.020}, \cite[Lemma 3.3]{KPV1993Duke}, \eqref{8.05} and Lemma 4.4, we have
\begin{eqnarray}
&&\left\|v_{1}\right\|_{X_{s+a,b,\phi}^{T_{\omega}}}=
\left\|\int_{0}^{t}V(t-s)\partial_{x}\left[\left(v_{1}+V(t)g+w_{2}\right)^{3}\right]ds\right\|_{X_{s+a,\frac{1}{2}+\epsilon,\phi}^{T_{\omega}}}\nonumber\\
&&\leq CT_{\omega}^{-\epsilon}\left\|\partial_{x}\left[\left(V(t)g+v_{1}+w_{2}\right)^{3}\right]
\right\|_{X_{s+a,-\frac{1}{2}+\epsilon,\phi}^{T_{\omega}}}\nonumber\\
&&\leq CT_{\omega}^{-\epsilon+\frac{1}{2}-\epsilon-b_{1}-\epsilon}\left\|\partial_{x}\left[\left(V(t)g+v_{1}+w_{2}\right)^{3}\right]
\right\|_{X_{s+a,-b_{1},\phi}^{T_{\omega}}}\nonumber\\
&&=C\left\|\partial_{x}\left[\left(V(t)g+v_{1}+w_{2}\right)^{3}\right]\right\|_{X_{s+a,-b_{1},\phi}^{T_{\omega}}}\nonumber\\
&&\leq C\left(\|V(t)g\|_{X_{s,b_{1},\phi}^{T_{\omega}}}^{3}+\|v_{1}\|_{X_{s,b_{1},\phi}^{T_{\omega}}}^{3}+\left\|\psi(t) w_{2}\right\|_{X_{s,b_{1},\phi}}^{3}\right)\nonumber\\
&&\leq C\left(1+\|g\|_{H^{s}}^{3}+\left\|\psi(t) w_{2}\right\|_{X_{s,b_{1},\phi}}^{3}\right)<\infty, \,a.s..\label{9.01}
\end{eqnarray}
From \eqref{9.01}, we have that \eqref{1.010} is valid.

The proof of Theorem 1.5 is completed.

\bigskip

\section{Proof of Theorem  1.6}

\setcounter{equation}{0}

 \setcounter{Theorem}{0}

\setcounter{Lemma}{0}

\setcounter{Proposition}{0}

\setcounter{section}{10}
In this section, we prove Theorem 1.6.

\noindent By using \cite[Theorem 1.1]{YYY2026}, \eqref{9.01} and $a=2s-\frac{1}{2}$, $s>\frac{1}{3}$, we have
\begin{eqnarray}
&&\|v_{1}\|_{C([0,T_{\omega}];L_{x}^{\infty})}\leq C\left\|v_{1}\right\|_{X_{s+a,\frac{1}{2}+\epsilon,\phi}^{T_{\omega}}}\nonumber\\
&&\leq C\left(1+\|g\|_{H^{s}}^{3}+\left\|\psi(t) w_{2}\right\|_{X_{s,b_{1},\phi}}^{3}\right)<\infty, \,a.s..\label{10.01}
\end{eqnarray}
From \eqref{10.01}, for almost every $\omega\in \Omega$, we have
\begin{eqnarray}
&&\lim_{t\rightarrow0}\|v_{1}\|_{L_{x}^{\infty}}=0.\label{10.02}
\end{eqnarray}
It follows from \eqref{10.02} that
\begin{eqnarray}
&&\mathbb{P}\left(\left\{\omega: \lim_{t\rightarrow0}\|v_{1}\|_{L_{x}^{\infty}}=0\right\}\right)=1.\label{11.03}
\end{eqnarray}

The proof of Theorem 1.6 is completed.

\bigskip

\section{Proof of Theorem  1.7}

\setcounter{equation}{0}

 \setcounter{Theorem}{0}

\setcounter{Lemma}{0}

\setcounter{Proposition}{0}

\setcounter{section}{11}
In this section, we prove Theorem 1.7.

\noindent By using \cite[Lemma 3.6]{YYY2026}, \eqref{9.01} and $a=2s-\frac{1}{2}$, $s>\frac{1}{3}$, we have
\begin{eqnarray}
&&\left\|v_{1}\right\|_{C([0,T_{\omega}];H^{\frac{1}{2}+\epsilon})}\leq C\left\|v_{1}\right\|_{X_{s+a,\frac{1}{2}+\epsilon,\phi}^{T_{\omega}}}\nonumber\\
&&\leq C\left(1+\|g\|_{H^{s}}^{3}+\left\|\psi(t) w_{2}\right\|_{X_{s,b_{1},\phi}}^{3}\right)<\infty, \,a.s..\label{11.01}
\end{eqnarray}
From \eqref{11.01}, for almost every $\omega\in \Omega$ and all $t\in[0,T_{\omega}]$, we have
\begin{eqnarray}
&&\lim_{|x|\rightarrow\infty}v_{1}=-\lim_{|x|\rightarrow\infty}\int_{0}^{t}V(t-s)\partial_{x}\left[(v_{1}+V(t)g+w_{2})^{3}\right]ds=0.\label{11.02}
\end{eqnarray}
It follows from \eqref{11.02} that
\begin{eqnarray}
&&\mathbb{P}\left(\left\{\omega: \forall t\in[0,T_{\omega}], \lim_{|x|\rightarrow\infty}v_{1}=0\right\}\right)=1.\label{11.03}
\end{eqnarray}

The proof of Theorem 1.7 is completed.

\bigskip

\leftline{\large \bf Acknowledgments}
Wei Yan is supported by the Natural Science Foundation of Henan province (No. 262300421062)
and the National Natural Science Foundation of China (Nos. 12371245 and 12571256).

\end{document}